\documentclass{amsart}
\usepackage{amsfonts}
\usepackage{amssymb}
\usepackage{times}

\def\esc#1{\langle #1\rangle}
\def\aut{\mathop{\hbox{aut}}}
\def\GSp{\mathop{\hbox{\rm GSp}}}

\def\Z{\mathbb Z}

\def\stab{\mathop{\hbox{\rm Stab}}}
\def\id{\mathop{\hbox{\rm id}}}
\def\aut{\mathop{\hbox{\rm Aut}}}
\def\ad{\mathop{\hbox{\rm ad}}}
\def\Spec{\mathop{\hbox{\rm Spec}}}

\def\T{\mathop{\mathcal T}}

\def\Z{\mathbb Z}
\def\I{\mathbf{i}}
\def\span#1{\langle #1\rangle}

\newtheorem{pr}{Proposition}

\newtheorem{lemma}{Lemma}
\newtheorem{de}{Definition}
\newtheorem{teo}{Theorem}
\newtheorem{example}{Example}

\def\Black{} 
\newfont{\hueca}{msbm10}
\def\hu #1{\hbox{\hueca #1}}\def\hu #1{\hbox{\hueca #1}}

\begin{document}

 \title{Gradings and
Symmetries  on Heisenberg type algebras}

\thanks{The authors are partially supported by the MCYT grant MTM2010-15223 and by the Junta de Andaluc\'{\i}a grants FQM-336, FQM-2467, FQM-3737.
The first and fourth authors are also supported by the PCI of the UCA `Teor\'\i a de
Lie y Teor\'\i a de Espacios de Banach' and by the PAI with
project number FQM-298.}

\author[A. J. Calder\'on]{Antonio Jes\'us Calder\'on Mart\'{\i}n}
\email{ajesus.calderon@uca.es}
\address{Dpto. Matem\'aticas\\Facultad de Ciencias, Universidad de C\'adiz\\
Campus de Puerto Real, 11510, C\'adiz, Spain}

\author[C. Draper]{Cristina Draper Fontanals}
\email{cdf@uma.es}
\address{Dpto. Matem\'atica Aplicada\\Escuela de las Ingenier\'{\i}as, Universidad de M\'alaga\\
Campus de Teatinos (Ampliaci\'{o}n), 29071, M\'alaga, Spain}

\author[C. Mart\'{\i}n]{C\'andido Mart\'{\i}n Gonz\'alez}
\email{candido@apncs.cie.uma.es}

\author[J. M. S\'anchez]{Jos\'e Mar\'{i}a S\'anchez Delgado}
\email{txema.sanchez@uma.es}
\address{Dpto. \'Algebra, Geometr\'\i a y Topolog\'\i a\\Facultad de Ciencias, Universidad de M\'alaga\\
Campus de Teatinos, 29080, M\'alaga, Spain}

\begin{abstract} We describe the fine (group) gradings on the Heisenberg algebras,
on the Heisenberg superalgebras  and on the twisted
Heisenberg
 algebras.
  We  compute the Weyl
groups of   these gradings. Also the results obtained respect to
Heisenberg superalgebras are  applied to the study of Heisenberg
Lie color algebras.

\medskip



\Black
\end{abstract}

\maketitle

\section{Introduction}  

In the last years there has been an increasing interest in the
study of the group gradings on Lie theoretic structures. In the
case of Lie algebras, this study has been focused  
on the simple
ones. A recent exhaustive survey on the matter is \cite{Book}, here we can briefly mention that  the  (complex) finite-dimensional simple case has
been studied, among other authors,    by Bahturin, Elduque,
Havl\'{\i}\v{c}ek, Kochetov, Patera, Pelantov\'{a}, Shestakov,
Zaicev and Zassenhaus
   \cite{Refe1, Ivan2, Zaicev, Alberto,  LGII, Zass} in the classical case (\cite{Alberto} encloses ${\frak d}_4$), while
the exceptional cases ${\frak g}_2, {\frak f}_4$ and ${\frak d}_4$
 have been studied by Bahturin, Draper, Elduque, Kochetov, Mart\'{\i}n,
Tvalavadze and Viruel \cite{otrog2, F4, forumd4, G2, Koche}.  The
fine group gradings on the real forms of ${\frak g}_2$ and ${\frak
f}_4$, also simple algebras, have been  classified by the three
first authors \cite{real}. Gradings have also been considered  in certain $\Z_2$-graded structures (for instance superalgebras), so the three
first authors  have studied  the case of the Jordan superalgebra $K_{10}$
(see \cite{Kac}), and the second and third authors  together with
Elduque have classified the fine gradings on exceptional Lie superalgebras in \cite{super}.
In relation with other Lie structures, the Lie triple systems of
exceptional type have also been considered from the viewpoint of
gradings (see \cite{triples}).

There is not much work done in the field of gradings on non-simple Lie algebras. Some relevant references are
 \cite{Refe2} and \cite{Refe3}.
 Our work is one step further in this direction.
We are interested in studying the gradings on
a family of non-simple Lie algebras, superalgebras, color algebras
and twisted algebras, namely, the Heisenberg algebras (resp.
superalgebras,  Lie color algebras and twisted ones).  Note that
Heisenberg (super) algebras are nilpotent and twisted Heisenberg
algebras are solvable.

The Heisenberg family of algebraic structures  was introduced by A. Kaplain
in \cite{Kaplain}.  
In the simplest case, that of a Heisenberg algebra, the notion is related to Quantum Mechanics.
As it is well known, the Heisenberg Principle of Uncertainty implies the noncompatibility of position and momentum observables
acting on fermions. This noncompatibility reduces to noncommutativity of the corresponding operators.
If we represent by $x$ the operator associated to position and by $\frac{\partial}{\partial x}$ the one associated to momentum
(acting for instance on a space $V$ of differentiable functions of a single variable), then $[\frac{\partial}{\partial x},x]=1_V$
which is nonzero. Thus we can identify the subalgebra generated by $1,x$ and $\frac{\partial}{\partial x}$
with the three-dimensional Heisenberg Lie algebra whose multiplication table in the basis $\{1,x,\frac{\partial}{\partial x}\}$
has as unique nonzero product: $[\frac{\partial}{\partial x},x]=1_V$.

The literature about Heisenberg structures is ample.
First, they have played an important role in Quantum
Mechanics,   where for instance twisted Heisenberg algebras
appear by a quantizing process from the classical Heisenberg
algebra $H(4)$. Thus, in
 \cite{Abdesselam},  coherent states for
power-law potentials are constructed by using generalized
Heisenberg algebras, being also shown that these coherent states
are useful for describing the states of real and ideal lasers
(\cite{Berrada})
or where a deformation of a Heisenberg algebra it is used to
describe the solutions of the $N$-particle rational Calogero model
and to solve the problem of proving the existence of supertraces
\cite{Konstein}.

Second, these structures are also important in  Differential Geometry.
The quotient Lie group $H_3( \mathbb{R})/H_3( \mathbb{Z})$, which is a compact
 smooth manifold without boundary
of dimension $3$ is, as mentioned by \cite{Semmes},
   one of the basic building blocks for $3$-manifolds studied in \cite{Thurston}.
   Our definition of twisted Heisenberg algebra is precisely motivated by this connection with Differential Geometry. In that context, the twisted Heisenberg algebra is the tangent algebra of a twisted Heisenberg group (certain semidirect product of a Heisenberg group with the real numbers). Its importance is due to a series of results (which can be consulted in \cite{Adams}), that describe the groups which can act by isometries in a
Lorentzian manifold. For instance:

{\bf Theorem.} {\cite[Theorem 11.7.3]{Adams}\sl \ Let $M$ be a
compact connected Lorentzian manifold   and $G$ a connected Lie
group acting isometrically and locally faithfully on $M$. Then its
Lie algebra $\mathfrak{g} = \mathfrak{k} \oplus \mathfrak{a}
\oplus \mathfrak{s}$ is a direct sum of a compact semisimple Lie
algebra $\mathfrak{k}$, an abelian algebra $\mathfrak{a}$ and a
Lie algebra $\mathfrak{s}$, which is either trivial, or isomorphic
to $\mathfrak{aff}(\mathbb{R})$, to a Heisenberg algebra $H_n$, to
a  twisted   Heisenberg algebra $H_n^\lambda $ with $\lambda \in
\mathbb{Q}_+^{(n-1)/2}$, or to $\mathfrak{sl}_2(\mathbb{R})$.}
\smallskip

Moreover, according to \cite[Chapter~8]{Adams}, the converse of
this result is also true: if $G$ is a connected simply connected
Lie group whose Lie algebra $\mathfrak{g} = \mathfrak{k} \oplus
\mathfrak{a} \oplus \mathfrak{s}$ is as above, then there is a
locally faithful isometric action of $G$ on a compact connected
Lorentz manifold. The physical meaning of this result  claims that
if you want a Lie group to act on a compact connected Lorentzian
manifold, then you must be ready to admit that the group may have
a Heisenberg section. More results about the relationship between
twisted Heisenberg algebras and Lorentzian manifolds can be found
in \cite{tesis}.

 Third, there are a variety of algebraic 
 works about these structures. For instance the set of
superderivations of a Heisenberg  superalgebra is applied to the
theory of cohomology, (\cite{Camacho}),
 the derivation
algebra of a Heisenberg Lie color algebra is computed, being a
simple complete Lie algebra  (\cite{chinos}).
Other  works  in this line are \cite{Konstantina, Zhou}.\smallskip

The study of Weyl groups of Lie gradings was inaugurated  
by
Patera and Zassenhaus in \cite{Zass}. Some concrete examples were
developed, for instance, in \cite{checos}. Recently, Elduque and
Kochetov have determined the Weyl groups of the fine gradings on
matrix, octonions, Albert and simple Lie algebras of types $A$,
$B$, $C$ and $D$ (see \cite{Weyl1, Weyl2}). The (extended) Weyl
group of a simple Lie algebra is the Weyl group of the Cartan
grading on that algebra (which is of course fine). Thus the notion
of Weyl group of a grading encompasses that of the usual Weyl
group with its countless applications. This is one of the reasons
motivating the study of Weyl groups on fine Lie gradings. On the
other hand, if we consider the category of graded Lie algebras
(not in the sense of Lie superalgebras), then the automorphism
group $G$ of such an object is defined as the group of
automorphisms of the algebra which preserve the grading, and the
Weyl group of the grading is an epimorphic image of $G$. Thus the
symmetries of a graded Lie algebra are present in the Weyl group
of the grading. Our work includes the description of the Weyl
group of the fine gradings on Heisenberg algebras, on Heisenberg
superalgebras and   on twisted Heisenberg algebras.

Over algebraically closed fields of characteristic zero, the study of gradings is strongly related to that of automorphisms. Although we get our classifications of gradings in Heisenberg algebras and on Heisenberg superalgebras in a  more general  context, we also compute the corresponding automorphism groups.
We would like to mention the work  \cite{Saal} which contains a
detailed study of the automorphisms on Heisenberg-type algebras.
So it has
been illuminating though our study of gradings goes a step
further. \smallskip

Finally, we devote some words to the distribution of results in our
work.
After a background on gradings in Section~2, we study the
group gradings on Heisenberg algebras in Section~3, by showing
that all of them are toral and by computing the Weyl group of the
only (up to equivalence)  fine one.
In Section~4 we study the fine
group gradings on Heisenberg superalgebras for $\mathbb{F}$ an algebraically closed field of characteristic different from $2$, and calculate their Weyl
groups.
In Section~5 we discuss on the concept of
Heisenberg Lie color algebra, give a description of the same and
show how the results in the previous section can be applied to
classify  a certain family of Heisenberg Lie color algebras.
Finally, in Section~6, we devote some attention to the concept of
twisted Heisenberg algebras and also compute their   group
gradings for $\mathbb{F}$   algebraically closed   of characteristic zero. We classify the  fine gradings up to equivalence and find their symmetries, which turn out to be very abundant.

\section{Preliminaries}

Throughout this work the base field   will be denoted by $\mathbb{F}$.
Let $A$ be an
algebra over $\mathbb{F}$. A {\it grading} on $A$ is a
decomposition $$\Gamma: A=\bigoplus\limits_{s \in S} A_s$$ of $A$
into   direct sum of nonzero subspaces such that for any $s_1,s_2
\in S$ there exists   $s_3 \in S$ such that $A_{s_1} A_{s_2}
\subset A_{s_3}$. The grading $\Gamma$ is said to be a {\it group
grading} if there is a group $G$ containing $S$ such that $A_{s_1}
A_{s_2} \subset A_{s_1 s_2}$ (multiplication of indices in the
group $G$) for any $s_1,s_2 \in S$. Then we can write $$\Gamma:
A=\bigoplus\limits_{g \in G} A_g,$$ by setting  $A_g=0$ if $g\in
G\setminus S$. In this paper all   the gradings we  consider will
be group gradings  where $G$ is a finitely generated abelian group
and $G$ is generated by the set of all the elements $g \in G$ such
that $A_g \ne 0$, usually called the \emph{support} of the grading
(the above $S$).

Given two gradings $A = \oplus_{g \in G}U_g$ and $A = \oplus_{h
\in H}V_h$, we shall say that they are {\em isomorphic} if there
is a group isomorphism $\sigma \colon G \to H$ and an (algebra)
automorphism $\varphi \colon A \to A$ such that $\varphi(U_g) =
V_{\sigma(g)}$ for all $g \in G$. The above two gradings are said
to be {\em equivalent} if there are    a bijection $\sigma \colon
S \to S'$ between the supports of the first and second gradings
respectively  and an algebra automorphism $\varphi$ of $A$ such
that $\varphi(U_g) = V_{\sigma(g)}$ for any $g \in S$.

Let $\Gamma$ and $\Gamma^{\prime}$  be two gradings on $A$. The
grading $\Gamma$ is said to be a {\it refinement} of
$\Gamma^{\prime}$ (or $\Gamma^{\prime}$ a {\it coarsening} of
$\Gamma$) if each homogeneous component of $\Gamma^{\prime}$ is a
(direct) sum of some homogeneous components of $\Gamma$. A grading
is called {\it fine} if it admits no proper refinements. A
fundamental concept to obtain the coarsenings of a given grading
is the one of universal grading group. Given a grading $\Gamma:{
{A}}=\oplus_{g\in G}{ {A}}_g$, one may consider the abelian group
${\tilde G}$ generated by the support of $\Gamma$  subject only to
the relations $g_1g_2=g_3$ if $0 \neq [{A}_{g_1},{A}_{g_2}]
\subset {A}_{g_3}.$ Then $A$ is graded over ${\tilde G}$; that is
$\tilde\Gamma: { {A}}=\oplus_{{\tilde g}\in {\tilde G}}{
{A}}_{\tilde{g}}$, where ${ {A}}_{\tilde{g}}$ is the sum of the
homogeneous components $A_g$ of $\Gamma$  such that the class of
$g$   in ${\tilde G}$ is ${\tilde g}.$  Note that  there is at
most one such homogeneous component and that this ${\tilde
G}$-grading $\tilde\Gamma$ is equivalent to $\Gamma$, since
$G\hookrightarrow\tilde G$, $g\mapsto\tilde g$ is an injective map
(not homomorphism). This group ${\tilde G}$ has the following
universal  property: given any coarsening $A=\oplus_{h\in H}{
{A}'}_h$ of $\tilde\Gamma$, there exists a unique group
epimorphism $\alpha\colon {\tilde G} \to H$ such that
$$A'_h=\bigoplus\limits_{{\tilde g} \in \alpha^{-1}(h)}{
{A}}_{{\tilde g}}.$$ The group ${\tilde G}$ is called the {\it
universal grading group} of $\Gamma$.  Throughout this
paper, the gradings will be considered over their universal
grading groups.
%

For a grading $\Gamma : A = \oplus_{g \in G}A_g$, the
\emph{automorphism group} of $\Gamma$, denoted $\aut(\Gamma)$,
consists of all self-equivalences of $\Gamma$, i.e., automorphisms
of $A$ that permute the components of $\Gamma$. The
\emph{stabilizer} of $\Gamma$, denoted $\stab(\Gamma)$, consists
of all automorphisms of the graded algebra $A$, i.e.,
automorphisms of $A$ that leave each component of $\Gamma$
invariant. The quotient group $\aut(\Gamma)/\stab(\Gamma)$ will be
called the \emph{Weyl group} of $\Gamma$ and denoted by
$\mathcal{W}(\Gamma)$.\smallskip


It is   well-known that any  finitely generated abelian
subgroup of diagonalizable automorphisms of $\aut(A)$ induces by
simultaneous diagonalization a grading on $A$.
When the field $\mathbb{F}$ is algebraically closed of  characteristic zero,  the converse is true:
 any grading on $A$ (finite dimensional algebra) is induced by a
finitely generated abelian subgroup of diagonalizable
automorphisms of $\aut(A)$.
The set of automorphisms inducing the grading (as simultaneous
diagonalization) is contained in the normalizer of some maximal
torus of the automorphism group (by \cite[Theorem~6 and
Theorem~3.15, p.~92]{Platonov}).
 A special kind of gradings arises
when we consider the inducing automorphisms
 not only in the normalizer of a maximal torus, but
 in a torus. Indeed,
a grading of an algebra $A$ is said to be {\em toral} if it is
produced by automorphisms within a torus of the automorphism group
of the algebra.
In the case of a fine grading, the grading is toral if and only if   
its universal grading group  is torsion-free, since  the universal grading group
is always isomorphic to the group of characters of the group of automorphisms inducing the grading.\smallskip


Next, let us recall the situation for superalgebras.
If  $L=L_0\oplus L_1$ is a  Lie
superalgebra over  $\mathbb{F}$ and $G$ a finitely generated
abelian group, a $G$-{\rm grading} on $L$ is a decomposition $
\Gamma:  L=\bigoplus_{g \in G} \big((L_0)_g \oplus (L_1)_g \big) $
where any  $(L_i)_g$ is a linear subspace of $L_i$ and where $
[(L_i)_{g_1},(L_j)_{g_2} ] \subset (L_{i+j})_{g_1g_2 } $ holds for
any $g_1,g_2  \in G$ and any $i,j\in\{0,1\}$ (sum modulo $2$).
Here the support of  $\Gamma$ is $\{g\in G\colon
{(L_i)}_g\ne0\text{ for some $i$}\}$ and
 everything  works  analogously to the case of a Lie algebra    
 with a grading. Note only a subtle difference: assuming that $L$ is  non-abelian,
   the trivial grading, $L=L_0\oplus L_1$, has as universal grading group $\Z_2$, 
while the trivial grading on  $L$ as a Lie algebra
  has the  trivial group as the universal grading group.\smallskip

Finally we give two fundamental lemmas  of purely geometrical
nature that will be applied in future sections. Recall that a
symplectic space $V$ is a linear space  provided with an
alternate nondegenerate bilinear form $\esc{\cdot,\cdot}$, and
that  in the finite-dimensional case  a standard result 
states the
existence of a \lq\lq symplectic basis\rq\rq, that is, a basis:
$\{u_1,u_1',\ldots,u_n,u_n'\}$ such that $\esc{u_i,u_i'}=1$ while
any other inner product is zero.

\begin{lemma}\label{avant}
 Let $(V,\esc{\cdot,\cdot})$ be a finite-dimensional symplectic space and assume that $V$ is the direct sum of linear subspaces $V=\oplus_{i\in I} V_i$ where for any $i\in I$ there is a unique $j\in I$
 such that $\esc{V_i,V_j}\ne 0$.
 Then there is a basis $\{u_1,u_1',\ldots,u_n,u_n'\}$
 of $V$ such that:
 \begin{itemize}
\item The basis is contained in $\cup_i V_i$.
\item For any $i,j$ we have $\esc{u_i,u_j}=\esc{u_i',u_j'}=0$.
\item For each $i$ and $j$ we have $\esc{u_i,u_j'}=\delta_{i,j}$ (Kronecker's delta).
\end{itemize}
\end{lemma}
\begin{proof}
 First we split $I$ into a disjoint union $I=I_1\cup I_2$ such
that $I_1$ is the set of all $i\in I$ such that $\esc{V_i,V_i}\ne
0$ and in $I_2$ we have all the indices $i$ such that there is  
$j\ne i$ with $\esc{V_i,V_j}\ne 0$. Now for each $i\in I_1$ the
space $V_i$ is symplectic with relation to the restriction of
$\esc{\cdot,\cdot}$ to $V_i$. So we fix in such $V_i$ a symplectic
basis. Take now $i\in I_2$ and let $j\in I$ be the unique index such
that $\esc{V_i,V_j}\ne 0$ (necessarily $j\in I_2$). Consider now the restriction
$\esc{\cdot,\cdot}\colon V_i\times V_j\to\mathbb{F}$. This map is
nondegenerate in the obvious sense (which implies
$\dim(V_i)=\dim(V_j))$. If we fix a basis $\{e_1,\ldots, e_q\}$ of
$V_i$, then by standard linear algebra arguments we get that there
is basis $\{f_1,\ldots f_q\}$ of $V_j$ such that $\esc{e_i,f_i}=1$
being the remaining inner products among basic elements zero.
Thus, putting together these basis suitable reordered we get the
symplectic basis whose existence is claimed in the Lemma.
\end{proof}

Since all the elements in the basis constructed above are in some component $V_i$ we will refer to this basis as a \lq\lq homogeneous basis\rq\rq\ of $V$.

\begin{lemma} \label{ovont} Let $(V,\esc{\cdot,\cdot})$ be a finite-dimensional linear space $V$ with a symmetric nondegenerate bilinear form $\esc{\cdot,\cdot}\colon V\times V\to\mathbb{F}$
where $\mathbb{F}$ is of characteristic other than $2$.
Assume
that $V=\oplus_{i\in I}V_i$ is the direct sum of linear subspaces
in such a way that  for each $i\in I$ there is a unique $j\in I$
such that $\esc{V_i,V_j}\ne 0$. Then  there is a basis
$B=\{u_1,v_1,\ldots, u_r,v_r,z_1,\ldots ,z_q\}$ of $V$ such that
\begin{itemize}
\item $B\subset\cup_i V_i$.
\item $\esc{z_i,z_i}\ne 0$, $\esc{u_i,v_i}=1$.
\item Any other inner product of elements in $B$ is zero.
\end{itemize}
\end{lemma}
\begin{proof}
 Let $I_1$ be the subset of $I$ such that for any $i\in I_1$ we have
$\esc{V_i,V_i}\ne 0$ and let $I_2$ be the complementary $I_2:=I\setminus
I_1$.
On one hand, each $V_i$  with $i\in I_1$  has an
orthogonal basis (char$\,\mathbb{F}\ne2$).
On the other hand,  for any $i\in I_2$ consider the unique $j\ne
i$ such that $\esc{V_i,V_j}\ne 0$. The couple $(V_i,V_j)$ gives a
dual pair $\esc{\cdot,\cdot}\colon V_i\times V_j\to\mathbb{F}$
(implying $\dim(V_i)=\dim(V_j))$ and for any  basis $\{e_k\}$ of
$V_i$, there is a dual basis $\{f_h\}$ in $V_j$ such that
$\esc{e_k,f_h}=\delta_{k,h}$ (Kronecker's delta). So putting
together all these bases (suitably ordered)  the required base on
$V$ is got.
 \end{proof}

Observe that if $\mathbb{F}$ is algebraically closed
the inner products $\esc{z_i,z_i}$ in Lemma \ref{ovont} may be chosen to be $1$. The basis constructed in Lemma~\ref{ovont} will be also termed a  \lq\lq homogeneous basis\rq\rq.

{We will have the occasion to use basic terminology of finite
groups}: $\Z_n$ for the cyclic
 group of order $n$, $S_n$
for the permutation group of $n$ elements and $D_n$ for the
dihedral group of order $2n$.

\section{Gradings on   Heisenberg algebras}\label{Sec_normalesHei}


A Lie algebra $H$ is called a {\it Heisenberg algebra}  if it is
nilpotent in two steps (that is, not abelian and $[[H,H],H]=0$)
with one-dimensional center ${\mathcal Z}(H):=\{x\in H :
[x,H]=0\}$. If $z\ne0$ is a fixed element in ${\mathcal Z}(H)$ and
we take $P$ any complementary subspace of ${\mathbb F}z$, then the
map $\langle \cdot , \cdot \rangle \colon P\times P\to \mathbb{F}$
given by $\langle v,v'\rangle z=[v,v']$ is a nondegenerate
skewsymmetric bilinear form, or, in other words, $(P,\langle \cdot
, \cdot \rangle)$ is a symplectic space. Of course, any Lie
algebra constructed from a symplectic space $(P,\langle \cdot ,
\cdot \rangle)$ as $H=P\oplus \mathbb{F}z$ with $z\in
\mathcal{Z}(H)$ and $[v,v']=\langle v,v'\rangle z$ for all
$v,v'\in P$, is a Heisenberg algebra. Recall that the dimension of
$P$ is necessarily even.


In particular, there is one Heisenberg algebra up to isomorphism
for each odd dimension $n=2k+1$, which we will denote $H_n$,
characterized by the existence of a basis
\begin{equation}\label{base1}
 B =
\{e_1,\hat{e}_1,\dots,e_k,\hat{e}_k,z\}
\end{equation}
  in which the
nonzero products  are $[e_i, \hat{e}_i]  = -[\hat{e}_i, e_i]=z$
for $1 \leq i \leq k$.


A fine grading on $H_n$ is obviously provided by this basis as:
\begin{equation}\label{lagradfinaH_n}
H_n= \langle e_1\rangle  \oplus \cdots \oplus \langle e_k\rangle
\oplus
 \langle \hat{e}_1\rangle  \oplus \cdots \oplus  \langle
\hat{e}_k\rangle \oplus\langle z\rangle,
\end{equation}
where each summand is a homogeneous  component.
This grading is also a group grading. For instance, it can be
considered as a $\mathbb{Z}$-grading by letting  $H_n =
\bigoplus\limits_{i=-k}^k (H_n)_i$ for
$$(H_n)_{-i} = \langle \hat{e}_i \rangle,\ (H_n)_0 = \langle z \rangle, \ (H_n)_i =
\langle e_i \rangle,\hspace{0.2cm}i=1,\ldots, k.$$
Moreover, this grading on $H_n$ is toral. It is enough to observe
that the group of automorphisms $\T$ which are represented by
scalar matrices relative to  the basis $B$ is a torus. In fact, it
is a maximal torus of dimension $k+1$. Indeed, an element $f\in
\T$ will be determined by the nonzero scalars
$(\lambda_1,\dots,\lambda_k,\lambda) \in \mathbb{F}^{k+1}$ such
that $f(z) = \lambda z$ and $f(e_i) = \lambda_i e_i$, being then
$f(\hat{e}_i) = \frac\lambda{\lambda_i} \hat{e}_i$. If we denote
such an automorphism by $t_{(\lambda_1,\dots,\lambda_k;\lambda)}$,
it is straightforward  that
$t_{(\lambda_1,\dots,\lambda_k;\lambda)}t_{(\lambda'_1,\dots,\lambda'_k;\lambda')}=t_{(\lambda_1\lambda'_1,\dots,\lambda_k\lambda'_k;\lambda\lambda')}$
and that any automorphism commuting with every element in $\T$
preserves the common diagonalization produced by $\T$, which is
precisely the one given by (\ref{lagradfinaH_n}). All this in
particular implies that this fine toral grading can be naturally
considered as a grading over the group $\mathbb{Z}^{k+1}$, which
is its universal grading group:
\begin{equation}\label{finasobregrupouniversal}
\begin{array}{ll}
\Gamma: & (H_n)_{(0,\dots,0;2)} = \langle z \rangle,\\
& (H_n)_{(0,\dots,1,\dots,0;1)} = \langle e_i \rangle, \  \text{ ($1$ in the $i$-th slot)}\\
& (H_n)_{(0,\dots,-1,\dots,0;1)} = \langle \hat{e}_i \rangle,\\
\end{array}
\end{equation}
if $i=1,\ldots, k$.

\smallskip

Our first aim is to prove that, essentially,  this  is the unique
fine  grading.

\begin{teo}\label{topa}
For any (group) grading on $H_n$, there is a basis $B=\{z,u_1,u_1',\ldots,u_n,u_n'\}$
of homogeneous elements of $H_n$ such that
$[u_i,u_i']=z$ and the remaining possible brackets among elements of $B$ are zero.

In particular,  there is only one fine grading up to equivalence on $H_n$, the
$\mathbb{Z}^{k+1}$-grading given by
(\ref{finasobregrupouniversal}). Thus, any  grading on $H_n$ is toral.
\end{teo}

\begin{proof}
As before we will consider the bilinear alternate form
$\esc{\cdot,\cdot}\colon H_n\times H_n\to \mathbb{F}$. In any
(group) graded Lie algebra the center admits a basis of
homogeneous elements,  so, if $H_n$ is graded by a group $G$, then there
is some $g_0\in G$ such that $z\in (H_n)_{g_0}$. Thus, denoting by
$\mathcal Z(H_n)=\mathbb{F}z$ the center of $H_n$, the (abelian) quotient
Lie algebra $P:=H_n/{\mathcal Z}(H_n)$ is a symplectic space
relative to $\esc{x+{\mathcal Z}(H_n), y+{\mathcal
Z}(H_n)}:=\esc{x,y}$. Denote by $\pi\colon H_n\to P$ the canonical
projection. By defining for each $g\in G$ the subspace
$P_g:=\pi((H_n)_g)$, it is easy to check that $P=\oplus_{g\in G}
P_g$ and that for any $g\in G$ there is a unique $h$,
($h=-g+g_0$),  such that $\esc{P_g,P_h}\ne 0$. Then Lemma
\ref{avant} provides a symplectic basis of $P$ of homogeneous
elements. Let $\{u_1+{\mathcal Z}(H_n),u_1'+{\mathcal
Z}(H_n),\ldots, u_n+{\mathcal Z}(H_n),u_n'+{\mathcal Z}(H_n)\}$ be
such basis (observe that each $u_i$ and $u_j'$ may be chosen in
some homogeneous component of $H_n$ being $u_i, u_j' \notin
\mathbb{F}z$). Then $B:=\{z,u_1,u_1',\ldots,u_n,u_n'\}$ is a basis
of homogeneous elements of $H_n$ such that $[u_i,u_i']=z$ and the
remaining possible brackets among basis elements are zero.

 To
finish the proof we can consider the maximal torus of $\aut(H_n)$ formed
by the automorphisms which are diagonal relative to
  the basis $B$. Up to conjugations, this torus is
formed by all the automorphisms
$t_{(\lambda_1,\ldots,\lambda_k;\lambda)}$ constructed above.
Since $B$ are eigenvectors for any of these elements, the initial grading is
toral and in fact it is a coarsening of the fine  grading described
in Equation~(\ref{finasobregrupouniversal}).
\end{proof}

In order to work  on the group of symmetries of this grading, the
Weyl group, we compute next the  automorphism group of the
Heisenberg algebra, $\aut(H_n)$. For any $f \in \aut(H_n)$ we have
$f(z) \in {\mathcal Z}(H_n)$ and so $f(z)= \lambda_f z$ for some
$\lambda_f \in \mathbb{F}^{\times}$. If we denote by $i \colon P
\to H_n$ the inclusion map, by $\pi \colon H_n \to P$ the
projection map and define $\bar{f} := \pi \circ f \circ i$, we
easily get that $\bar{f}$ is a linear automorphism of $P$
satisfying $\langle \bar{f}(x_P), \bar{f}(y_P) \rangle = \lambda_f
\langle x_P, y_P \rangle$ for any $x_P, y_P \in P$. Indeed, given
any $x = x_P + \lambda z,$ and $y = y_P + \mu z$ in $H_n, \lambda,
\mu \in \mathbb{F}$, we have, taking into account $[z,H_n]=0$,
that

$$f([x,y]) = f([x_P,y_P]) = f(\langle x_P, y_P \rangle z) = \langle x_P,
y_P \rangle f(z)=\langle x_P, y_P \rangle  \lambda_f z$$ and
%
$$[f(x),f(y)] = [\bar{f}(x_P), \bar{f}(y_P)] = \langle
\bar{f}(x_P), \bar{f}(y_P) \rangle z.$$
Hence $\bar{f}$ belongs to $\GSp(P) := \{g\in\mathop{\rm End}(P) :
\text{ there is $\lambda_g \in \mathbb{F}^{\times}$ with } \langle
g(x), g(y) \rangle = \lambda_g \langle x, y \rangle \  \forall
x,y\in P \}$, the {\it similitude group} of $(P, \langle \cdot,
\cdot \rangle)$.

As a consequence, an arbitrary $f \in \aut(H_n)$ has as associated
matrix relative to the basis $B$ in Equation~(\ref{base1}),

$$
M_{B}(f) = \tiny{ \left(%
\begin{array}{c|c}
M_{B_P}( \bar{f}) & 0 \\
 \hline \underline{\lambda} & \lambda_f \\
\end{array}%
\right)}
$$
for $B_P=B\setminus\{z\}$ a basis of $P$,  $\bar{f} = \pi \circ f \circ i \in \GSp(P)$ and for the vector
$\underline{\lambda}= (\lambda_1,..., \lambda_{2k}) \in
\mathbb{F}^{2k}$ given by $\lambda_{2i-1}z=(f-\bar f)(e_i)$ and
$\lambda_{2i}z=(f-\bar f)(\hat{e}_i)$ if $i\le k$. Moreover, if we
define in the set $\GSp(P) \times \mathbb{F}^{2k}$ the
(semidirect) product
$$
(\bar{f}, \underline{\lambda})(\bar{g}, \underline{\eta}) :=
(\overline{fg}, \underline{\lambda}M_{B_P}(\bar g) +
\lambda_f\,\underline{\eta}),
$$
it is straightforward to verify that the mapping $$\Omega \colon
\aut(H_n) \to \GSp(P) \ltimes \mathbb{F}^{2k}$$ given by
$\Omega(f) = (\bar{f}, \underline{\lambda})$ is a group
isomorphism.

Take, for each permutation $\sigma \in S_k$, the map $\tilde\sigma
\colon H_n \to H_n$ given by $\tilde\sigma(e_i) = e_{\sigma(i)}$,
$\tilde\sigma(\hat{e}_i) = \hat{e}_{\sigma(i)}$ and
$\tilde\sigma(z) = z$. It is clear that $\tilde\sigma$ is an
automorphism permuting the homogeneous components of $\Gamma$, the
grading in (\ref{finasobregrupouniversal}), that is, $\tilde\sigma
\in \aut(\Gamma)$.

Other remarkable elements in the automorphism group of the grading
are the following ones: for each index $i\le k$, take $\mu_i
\colon H_n \to H_n$ given by $\mu_i(e_i) = \hat{e}_i,
\mu_i(\hat{e}_i) = -e_i, \mu_i(e_j) = e_j, \mu_i(\hat{e}_j) =
\hat{e}_j$ (for $j \ne i$) and $\mu_i(z) = z$.

Denote by $[f]$ the class of an automorphism $f \in \aut(\Gamma)$
in the quotient $\mathcal{W}(\Gamma) =
\aut(\Gamma)/\stab(\Gamma)$.

\begin{pr}\label{pr_weylfina}
The Weyl group $\mathcal{W}(\Gamma)$ is generated by
$\{[\tilde\sigma]  : \sigma\in S_k\}$ and $[\mu_1]$.
\end{pr}

\begin{proof}
Let $f$ be an arbitrary element in $\aut(\Gamma)$. The elements in
$\aut(\Gamma)$ permute the homogeneous components of the grading
$\Gamma$, but $\mathbb{F}z$ remains always invariant. Thus
$f(e_1)$ belongs to some homogeneous component  different from
$\mathbb{F}z$, and there is an index $i\le k$ such that either
$f(e_1)\in \mathbb{F}^{\times}e_i$ or $f(e_1) \in
\mathbb{F}^{\times}\hat{e}_i$. We can assume that $f(e_1) \in
\mathbb{F}^{\times}e_i$ by replacing, if necessary, $f$ with
$\mu_if$. Now, take the permutation $\sigma = (1,i)$ which
interchanges $1$ and $i$, so that $f' = \tilde\sigma f$ maps $e_1$
into $\alpha e_1$ for some $\alpha\in\mathbb{F}^{\times}$. Note
that
$\alpha[e_1,f'(\hat{e}_1)]=[f'(e_1),f'(\hat{e}_1)]=f'([e_1,\hat{e}_1])
= f'(z)$ is a nonzero multiple of $z$, hence $f'(\hat{e}_1) \notin
\{x \in H_n  : [x,e_1] = 0\} = \langle z,e_1,e_i,\hat{e}_i :
2\le i\le k\rangle$. But $f' \in \aut(\Gamma)$, so $f'(\hat{e}_1)
\in \langle\hat{e}_1\rangle$. In a similar manner the automorphism
$f'$ sends $e_2$ to  some $\langle e_j\rangle$ or some $\langle
\hat{e}_j\rangle$ for $j \ne 1$, hence we can replace $f'$ by $f''
\in \{\widetilde{(2,j)}f',\widetilde{(2,j)}\mu_jf'\}$ such that
$f''$ preserves the homogeneous components $\langle e_1\rangle,
\langle \hat{e}_1\rangle$, $\langle e_2\rangle$, $\langle
\hat{e}_2\rangle$ and $\langle z\rangle$. By arguing as above, we
can multiply $f$ by  an element in the subgroup generated by $
\mu_j$ and $\tilde\sigma$, for $1\le j\le k$ and $\sigma \in S_k$,
such that the product stabilizes all the components, so that it
belongs to $\stab(\Gamma)$. The proof finishes if we observe that
$\tilde\sigma\mu_i = \mu_{\sigma(i)}\tilde\sigma$ for all $i\le
k$, so that all $\mu_i$'s belong to the noncommutative group
generated by $\{[\tilde\sigma]  : \sigma\in S_k\}$ and $[\mu_1]$.
\end{proof}

Hence $\mathcal{W}(\Gamma) = \{[\mu_{i_1} \dots
\mu_{i_s}\tilde\sigma]  : \sigma\in S_k,\,1\le i_1\le\dots\le
i_s\le k\}$ has $2^kk!$ elements. Observe that, although any
$\mu_i$ has order 4, its class $[\mu_i]$ has order 2. Besides
$\mu_i$ and $\mu_j$ commute, so we can identify
$\mathcal{W}(\Gamma)$ with the group $\mathcal{P}(K)\rtimes S_k$,
where the product is given by $(A,\sigma)(B,\eta) =
(A\bigtriangleup\sigma(B),\sigma\eta)$ if $A, B \subset
K=\{1,\dots,k\}$, $\sigma,\eta\in S_k$, and where the elements in
$\mathcal{P}(K)$ are the subsets of $\{1,\dots,k\}$ and
$\bigtriangleup$ denotes the symmetric difference. Thus
$$
\mathcal{W}(\Gamma)\cong\mathbb{Z}_2^k\rtimes S_k.
$$

\section{Gradings on  Heisenberg superalgebras}\label{Sec_superHei}

In this section we will assume the ground field $\mathbb{F}$ to be
algebraically closed and of
characteristic other than $2$.
A {\rm Heisenberg superalgebra} $\mathcal{H}=\mathcal{H}_0\oplus
\mathcal{H}_1$ is a  nilpotent in two steps Lie superalgebra with
one-dimensional center such that $[\mathcal{H}_0,\mathcal{H}_1] =
0$. In particular this implies that the even part is a Heisenberg
algebra, so that it is determined up to isomorphism by its
dimension. Note that, if $x,y \in \mathcal{H}_1$, then $[x,y] =
[y,x] \in \mathbb{F}z = \mathcal{Z}(\mathcal{H})$, so there is a
nondegenerate symmetric bilinear form $\langle \cdot , \cdot
\rangle \colon \mathcal{H}_1 \times \mathcal{H}_1 \to \mathbb{F}$
 such that $[x,y] = \langle x,y\rangle z$ for all $x,y \in
\mathcal{H}_1$. Hence there is a basis of $\mathcal{H}_1$ in which
the matrix of $\langle \cdot , \cdot \rangle$ is the identity
matrix. If $B_1 = \{w_1,...,w_m\}$  denotes such a basis, and $B_0
= \{z,e_1,\hat{e}_1,...,e_k,\hat{e}_k\}$ denotes the basis of
$\mathcal{H}_0$ as in Section~\ref{Sec_normalesHei}, then the
product is given by
$$\begin{array}{l}
[{e}_i,\hat e_i] =-[\hat {e}_i,e_i]= z, \hspace{0.3cm} 1 \leq i \leq k,\\
{[w_j, w_j]} =  z, \hspace{0.3cm} 1 \leq j \leq m,
\end{array}$$
with all other  products being zero. As this algebra is completely
determined by $n=2k+1$ and $m$, the dimensions of the even and odd
part respectively, we denote it by $H_{n,m}$.

If we denote by
$$\circ \colon (H_{n,m})_1 \times (H_{n,m})_1 \to (H_{n,m})_0$$ the
bilinear mapping $x_1 \circ y_1 := [x_1,y_1]$, we get that the
product in $H_{n,m}$ can be expressed  by
$$[(x_0,x_1), (y_0,y_1)] = [x_0,y_0] + x_1 \circ y_1,$$
for all $x_i,y_i \in (H_{n,m})_i$.


\medskip

%


Assume now that $H_{n,m}$ is a graded superalgebra and $G$ is the
grading group. Then the even part $(H_{n,m})_0$ admits a basis
$\{z,e_1,\hat e_1, \ldots e_k, \hat{e}_k\}$ of homogeneous
elements as has been proved in the previous section.  On the other
hand the product in $(H_{n,m})_1$ is of the form $x\circ
y=\esc{x,y}z$ for any $x,y\in (H_{n,m})_1$ and where
$\esc{\cdot,\cdot}\colon (H_{n,m})_1\times
(H_{n,m})_1\to\mathbb{F}$ is a symmetric  nondegenerate bilinear
form. Furthermore, for any $g\in G$, denote by $\mathcal{L}_g$ the
subspace $\mathcal{L}_g:=(H_{n,m})_1\cap (H_{n,m})_g$ (that is, the
odd part of the homogeneous component of degree $g$ of $H_{n,m}$).
Then $(H_{n,m})_1=\oplus_{g\in G}\mathcal{L}_g$ is a decomposition
on linear  subspaces and for any $g\in G$ there is a unique $h\in
G$ such that $\esc{\mathcal{L}_g,\mathcal{L}_h}\ne 0$: indeed,
assume $\esc{\mathcal{L}_g,\mathcal{L}_h}\ne 0$. Then $0\ne
\mathcal{L}_g\circ \mathcal{L}_h\subset\mathbb{F}z$ and this implies
that $g+h=g_0$ where $g_0$ is the degree of $z$. Thus $h=g_0-g$ is
unique. Next we apply Lemma \ref{ovont} to get a basis
 $\{z,e_1,\hat{e}_1,\ldots
e_k,\hat{e}_k,u_1,v_1,\ldots, u_r,v_r,z_1,\ldots,z_q\}$ of
$H_{n,m}$ (of homogeneous elements) such that $z, e_i, \hat e_i$ generate the even part of $H_{n,m}$ while $u_1,v_1,\ldots, u_r,v_r,z_1,\ldots,z_q$ generate the odd part, and the nonzero products are:
\begin{equation}\label{relacionesparalasfinas_super}
[e_i,\hat e_i]=[u_j,v_j]=[z_l,z_l]=z,
\end{equation}
for $i\in\{1,\ldots,k\}$, $j\in\{1,\ldots,r\}$ and $l\in\{1,\ldots q\}$.

 This basis provides a
$\mathbb{Z}^{1+k+r} \times \mathbb{Z}_2^{m-2r}$-grading on
$H_{n,m}$ given by
\begin{equation}\label{gradsuper}
\begin{array}{ll}
\Gamma^r: & (H_{n,m})_{(2;0,\dots,0;0,\dots,0;\bar 0,\dots,\bar0)}=\langle z\rangle, \\
& (H_{n,m})_{(1;0,\dots,1,\dots,0;0,\dots,0;\bar 0,\dots,\bar0)}=\langle e_i\rangle, \\
& (H_{n,m})_{(1;0,\dots,-1,\dots,0;0,\dots,0;\bar 0,\dots,\bar0)}=\langle \hat{e}_i\rangle, \\
& (H_{n,m})_{(1;0,\dots,0;0,\dots,1,\dots,0;\bar 0,\dots,\bar0)}=\langle u_j\rangle, \\
& (H_{n,m})_{(1;0,\dots,0;0,\dots,-1,\dots,0;\bar 0,\dots,\bar0)}=\langle v_j\rangle, \\
& (H_{n,m})_{(1;0,\dots,0;0,\dots,0;\bar
0,\dots,\bar1,\dots,\bar0)}=\langle z_l\rangle,
\end{array}
\end{equation}
if $i\le k,j\le r,l\le q$. This grading is a refinement of the
original $G$-grading of the algebra.

Observe that for each $r$ such that $0\le 2r\le m$, there exists a
basis of $(H_{n,m})_1$ satisfying the
relations~(\ref{relacionesparalasfinas_super}) by taking for $j\le
r,l\le m-2r$,
$$
\begin{array}{l}
u_j := \frac1{\sqrt2}(w_{2j-1}+\I w_{2j}),\\
v_{j} := \frac1{\sqrt2}(w_{2j-1}-\I w_{2j}),\\
z_l := w_{l+2r},
\end{array}
$$
if $\I\in\mathbb{ F}$ is a primitive  square root of the unit. If
the starting $G$-grading is fine, then it is equivalent  to the
$\mathbb{Z}^{1+k+r} \times \mathbb{Z}_2^{m-2r}$-grading $\Gamma^r$
provided by the above basis. Therefore we have  proved the
following result.

\begin{teo}\label{superteo}
Up to equivalence, there are $\frac{m}{2}+1$ fine gradings on
$H_{n,m}$ if $m$ is even and $\frac{m+1}{2}$ in case $m$ is odd,
namely, $\{\Gamma^r : 2r\le m\}$. All of these are inequivalent,
and only one is toral, $ \Gamma^{m/2}$ when $m$ is even.
\end{teo}

The inequivalence follows easily from the fact that the universal grading groups are not isomorphic.
Note that Theorem~\ref{topa} is a particular case of Theorem~\ref{superteo} for $m=0$.\smallskip


Next, let us compute the group $\aut(H_{n,m})$ of automorphisms of
$H_{n,m}$. Recall that $\aut(H_{n,m})$ is formed by the linear
automorphisms $f \colon H_{n,m} \to H_{n,m}$ such that $f([x,y]) =
[f(x),f(y)]$ for all $x, y \in H_{n,m}$ and $f((H_{n,m})_i) =
(H_{n,m})_i$ for $i \in \{0, 1\}$.

If we identify $(H_{n,m})_1$ with the underlying vector space
endowed with the inner product $\langle\ ,\ \rangle$, then we can
consider the group $\hbox{GO}((H_{n,m})_1)=\{g \in {\rm
End}((H_{n,m})_1) : \langle g(x),g(y) \rangle = \lambda \langle
x,y \rangle \,\forall x,y \in (H_{n,m})_1$ for some $\lambda \in
{\mathbb F}^{\times}\}$, and the group homomorphism
$\pi\colon\aut(H_{n,m})\to \hbox{GO}((H_{n,m})_1)$ such that
$\pi(f)=f\vert_{(H_{n,m})_1}$. We prove next that $\pi$ is an
epimorphism. Given $f\in\hbox{GO}((H_{n,m})_1)$, we know that there
is $\lambda_f\in \mathbb{F}^{\times}$ such that $\langle
f(x),f(y)\rangle=\lambda_f\langle x,y\rangle$. Then we can define
the linear map
 $\hat f\colon H_{n,m}\to H_{n,m}$ such that $\hat f$ restricted to
$(H_{n,m})_1$ is $f$ and $\hat f(z):=\lambda_f z$, $\hat
f(e_i):=\lambda_f e_i$ and $\hat f(\hat e_i):=\hat e_i$ for any
$i$. Then $\hat f\in\hbox{aut}(H_{n,m})$ and $\pi(\hat f)=f$.
Furthermore $\ker(\pi)\cong\hbox{Sp}(P)\times {\hbox{\hu
F}}^{2k}$, taking into account the results in
Section~\ref{Sec_normalesHei} and that any $f\in
\hbox{aut}(H_{n,m})$ such that $\pi(f)=f\vert_{(H_{n,m})_1}=1$
satisfies that $f(z)=z$. Therefore we have a short exact sequence
$$1\to \hbox{Sp}(P)\ltimes {\hbox{\hu F}}^{2k}\buildrel i\over
\longrightarrow\aut(H_{n,m}) \buildrel \pi\over \longrightarrow
\hbox{GO}((H_{n,m})_1)\to 1$$ which is split since $j\colon
\hbox{GO}((H_{n,m})_1)\to \aut(H_{n,m})$ defined by $j(f):=\hat f$
satisfies $\pi j=1$. Then $\aut(H_{n,m})$ is the semidirect
product $$\aut(H_{n,m})\cong (\hbox{Sp}(P)\ltimes {\hbox{\hu
F}}^{2k})\rtimes \hbox{GO}((H_{n,m})_1).$$ \smallskip


Finally, we compute the Weyl groups of the  fine gradings on $H_{n,m} $ described in Theorem~\ref{superteo}.
Consider, as in Section~\ref{Sec_normalesHei}, the maps
$\tilde\sigma,\mu_i\in\aut((H_{n,m})_0)$ if $\sigma\in S_k$, $i\le
k$, and extend
 to automorphisms of $H_{n,m}$ by setting  $\tilde\sigma\vert_{(H_{n,m})_1}=\mu_i\vert_{(H_{n,m})_1}=\id$.
Thus, $  \tilde\sigma,\mu_i  \in \aut(\Gamma^r)$, for $\Gamma^r$
the grading in (\ref{gradsuper}). Take too, for each permutation
$\sigma \in S_r$, the map $\bar\sigma \colon H_{n,m} \to H_{n,m}$
given by $\bar\sigma\vert_{(H_{n,m})_0}= \id$, $\bar\sigma(u_i) =
u_{\sigma(i)}$, $\bar\sigma(v_i) = v_{\sigma(i)}$ and
$\bar\sigma(z_l) = z_l$. Also consider for each permutation $\rho
\in S_q$, the map $\hat\rho \colon  H_{n,m}  \to  H_{n,m} $ given
by $\hat\rho\vert_{(H_{n,m})_0}= \id$, $\hat\rho(u_i) = u_i$,
$\hat\rho(v_i) = v_i$  and $\hat\rho(z_l)=z_{\rho(l)} $. Finally
consider for each index $i \le r$, the map $\mu'_i \colon H_{n,m}
\to H_{n,m}$ given by $\mu'_i \vert_{(H_{n,m})_0}= \id    $,
$\mu'_i(u_i) = v_i, \mu'_i(v_i) = -u_i, \mu'_i(u_k) = u_k,
\mu'_i(v_k) = v_k$  for $k \ne i$  and $\mu'_i(z_l) = z_l$. It is
clear   that   $\bar\sigma, \hat\rho,\mu'_i \in \aut(\Gamma^r)$ in
any case.

\begin{pr}
The Weyl group $\mathcal{W}(\Gamma^r)$ is generated by $[\mu_1]$,
$[\mu'_1]$, $\{[\tilde \sigma] : \sigma\in S_k\}$, $\{[\bar
\sigma] : \sigma\in S_r\}$ and $\{[\hat \sigma] : \sigma\in
S_q\}$, with $k=(n-1)/2, q=m-2r$. 
%
\end{pr}

\begin{proof}
We know by Section~\ref{Sec_normalesHei} that the subgroup $W$ of
$\mathcal{W}(\Gamma^r)$ generated by the classes of the elements
fixing all the homogeneous components of $(H_{(n,m)})_1$  is
$\{[\mu_{i_1} \dots \mu_{i_s}\tilde\sigma]  : \sigma\in
S_k,\,1\le i_1\le\dots\le i_s\le k\}$.

Let $f$ be an arbitrary element in $\aut(\Gamma^r)$. As
$f\vert_{(H_{(n,m)})_0 }$ preserves the grading $\Gamma$ in
Equation~(\ref{finasobregrupouniversal}), then we can compose $f$
with an element in $W$ to assume that $f$ preserves all the
homogeneous components of $(H_{(n,m)})_0$.

Then the element $f(u_1)$ belongs to some homogeneous component of
$(H_{(n,m)})_1$, but it does not happen that there is $j \le q$
such that $f(u_1) \in \mathbb{F}z_j$, since then
$0=f([u_1,u_1])=z_j\circ z_j=z$. So  there is an index $i \le r$
such that  either $f(u_1) \in \mathbb{F}u_i$  or $f(u_1) \in
\mathbb{F}v_i$.
%
%
%
%
The same arguments as in the proof of
Proposition~\ref{pr_weylfina} show that we can replace $f$ by
$\mu'_{j_1} \dots \mu'_{j_s}\bar\sigma f$ for $1\le j_1\le\dots\le
j_s\le r$ and $\sigma\in S_r$ to get that $f(u_i)\in
\mathbb{F}u_i$ and $f(v_i)\in \mathbb{F}v_i$ for all $i\le r$.

Now it is clear that  $f(z_1) \in \mathbb{F}z_l$ for some $1 \le l
\le q$. If $l\ne1$, we can replace $f$ with $ \hat\rho f$, for
$\rho = (1,l)$, so that we can assume $f(z_1) \in \mathbb{F}z_1$.
And, in the same way, we can assume that $f(z_l) \in
\mathbb{F}z_l$ for $1 \le l \le q$. Our new $f$ belongs to
$\stab(\Gamma^r)$.

The proof finishes if we observe that $\bar\sigma\mu'_i =
\mu'_{\sigma(i)}\bar\sigma$ for all $i \le r$ and $\sigma\in S_r$,
and that
  $\bar\sigma$ as well as  $\mu'_i$ commute with $\hat \rho$ for all $\rho\in S_q$.
\end{proof}

Hence, an arbitrary element in $\mathcal{W}(\Gamma^r) $ is
$$
[\mu_{i_1} \dots \mu_{i_s}\tilde\sigma\mu'_{j_1} \dots
\mu'_{j_t}\bar\eta\hat\rho]
$$
for $1\le i_1\le\dots\le i_s\le k$, $1\le j_1\le\dots\le j_t\le
r$, $\tilde\sigma \in S_k$, $ \bar\eta\in S_r$, $\hat\rho \in
S_q$, so that  $\mathcal{W}(\Gamma^r)$ is isomorphic to
$$
   (\mathcal{P}(\{1,\dots,k\})\rtimes S_k)\times (\mathcal{P}(\{1,\dots,r\})\rtimes S_r)\times S_q
 $$
 with the product as in Section~\ref{Sec_normalesHei}, and in a more concise form,
 $$
 \mathcal{W}(\Gamma^r)\cong\mathbb{Z}_2^{r+k}\rtimes (S_k\times S_r\times S_q).
   $$


 \section{An application to Heisenberg Lie color algebras}

The base field $\mathbb{F}$ will be supposed
throughout  this section algebraically closed and of
characteristic other than 2, as in Section~4.  Lie color algebras were introduced
in \cite{1} as a generalization of Lie superalgebras and hence of
Lie algebras. This kind of algebras has attracted the interest of
several authors in the last years, (see \cite{Yo, Dmitri, Price,
Zhang, Xueme}),
 being also remarkable the important role they
play in theoretical physic, specially in conformal field theory
and supersymmetries (\cite{JMP2,JMP1}).

\medskip

\begin{de}
Let  $ G$ be an
 abelian  group. A {\rm
skew-symmetric bicharacter}  of   $G$ is a map $\epsilon\colon G
\times G \longrightarrow {\hu F}^{\times}$ satisfying

$$\epsilon(g_1,g_2)= \epsilon(g_2,g_1)^{-1},$$
$$
\epsilon(g_1,g_2+g_3)= \epsilon(g_1,g_2)\epsilon(g_1,g_3),
$$
for any $g_1,g_2,g_3 \in G$.
\end{de}

Observe  that $\epsilon(g,{0}) = 1$ for any $g \in G$, where $0$
denotes the identity element of $G$.

\begin{de}
Let $L = \bigoplus\limits_{g \in G}L_g$ be a $G$-graded ${\hu
F}$-vector space. For a nonzero homogeneous element $v \in L$,
denote by  $\deg{v}$     the unique  element in $G$ such that $v
\in L_{\deg{v}}$. We shall say that $L$ is a {\rm Lie color
algebra} if it is endowed with an  ${\hu F}$-bilinear map (the {\it
Lie color bracket})
$$[\cdot,\cdot]\colon L \times L \longrightarrow L$$ satisfying
$[L_g,L_h]\subset L_{g+h}$ for all $g,h\in G$ and

$$[v,w] = -\epsilon(\deg{v},\deg{w})[w,v], \hspace{0.2cm} \hbox{{\rm(color skew-symmetry)}}$$
$$[v,[w,t]] = [[v,w],t] + \epsilon(\deg{v}, \deg{w})[w,[v,t]], \hspace{0.2cm}  \hbox{{\rm(Jacobi color identity)}}$$

\noindent for all homogeneous elements $v,w,t \in L $ and for some
skew-symmetric bicharacter $\epsilon$.
\end{de}

Two  Lie color algebras are {\it isomorphic} if they are
isomorphic as graded algebras.

Clearly any Lie algebra is a Lie color algebra and also   Lie
superalgebras are examples of Lie color algebras  (take $G = {\hu
Z}_2$ and $\epsilon(i,j) = (-1)^{ij},$ for any $i,j \in {\hu
Z}_2).$

Heisenberg Lie color algebras have been previously considered in
the literature (see \cite{chinos}). Fixed a skew-symmetric
bicharacter $\epsilon\colon G \times G \longrightarrow {\hu
F}^{\times}$, a Heisenberg Lie color algebra $H$ is defined in
\cite{chinos}  as a $G$-graded vector space, where $G$ is a
torsion-free abelian group with a basis $\{ \epsilon_1,
\epsilon_2,..., \epsilon_n\}$, in the form
$$H=\bigoplus\limits_{i=1}^{n} {\hu F} p_i \oplus \bigoplus\limits_{j=1}^{n} {\hu F}
q_j \oplus  {\hu F}  c$$ where $p_i \in H_{\epsilon_i}$, $q_j \in
H_{-\epsilon_j}$, and $c \in H_0$; and where the Lie color bracket
is given by $[p_i, q_i]=\delta_{ij}c$ and  $[p_i, p_j]=[q_i,
q_j]=[p_i, c]=[c, p_i]=[q_j, c]=[c, q_j]=0$.

Observe that, following this definition,  the class of  Heisenberg
 superalgebras is not contained in the one of  Heisenberg Lie
color algebras. Hence, this definition seems to us very
restrictive. So let us briefly discuss about the concept of
Heisenberg Lie color algebras. Any  Heisenberg  algebra
(respectively Heisenberg  superalgebra)  $H$  is characterized
among the Lie algebras  (respectively among the Lie superalgebras)
for satisfying $[H,H]={\mathcal Z}(H)$ and $\dim ({\mathcal
Z}(H))=1$.  Hence, it is natural to introduce the following
definition.
\begin{de}\label{colorH}
A  Heisenberg Lie color algebra is  a Lie color algebra $L$ such
that $[L,L]={\mathcal Z}(L)$ and $\dim ({\mathcal Z}(L))=1$.
\end{de}

\noindent \textbf{Examples.}

\textbf{1.} The above so called  Heisenberg Lie color algebras,
given in \cite{chinos}, satisfy these conditions and so can be
seen as  a particular case of the ones given by
Definition~\ref{colorH}. As examples of Heisenberg Lie color
algebras we also have the Heisenberg   algebras ($G=\{0\}$) and
the Heisenberg  superalgebras ($G=\{  \mathbb{Z}_2\}$). We can
also consider any Heisenberg $G$-graded algebra as a Heisenberg
Lie color algebra for the group $G$ and the trivial bicharacter
given by $\epsilon(g,h)=1$ for all $g,h\in G$.

\medskip



\textbf{2.}  Any $G$-graded $L=\oplus_{g\in G} M_g$ Heisenberg
superalgebra $L=L_0\oplus L_1$
 gives rise to  a Heisenberg Lie
color algebra relative to the group $G \times {\mathbb Z}_2$ and
an adequate skew-symmetric bicharacter $\epsilon$. In fact, we
just have to $(G \times {\mathbb Z}_2)$-graduate  $L$ as $
L=\oplus_{(g,i)\in G\times {\mathbb Z}_2} M_{(g,i)}$ where
$M_{(g,i)}=M_g$ with $M_{(g,i)}\subset L_i$, and define $\epsilon:
(G\times {\mathbb Z}_2) \times (G\times {\mathbb Z}_2) \to {\hu
F}^{\times}$ as $\epsilon((g,i), (h,j)):= (-1)^{ij}$. Observe that
both gradings on $L$ are equivalent.

\medskip

\textbf{3.}
 Consider some $g_0 \in G $, a graded vector
space $V=\bigoplus\limits_{g \in G}V_g$ such that  $\dim
(V_g)=\dim (V_{-g+g_0})$ for any $g \notin \{g_0,0\}$ and $\dim
(V_{g_0})=\dim (V_{0})+1$ in case $g_0 \neq 0$; and any
skew-symmetric bicharacter $\epsilon \colon G \times G
\longrightarrow {\hu F}^{\times}$ satisfying $\epsilon(g,g)=-1$
for any $g \in G$ such that $2g=g_0$ and $V_g \neq 0$. Fix bases
 $\{z, u_{g_0,1},...,
u_{g_0,n_{g_0}}\}$ and $\{\hat{u}_{g_0,1},...,
\hat{u}_{g_0,n_{g_0}}\}$ of $V_{g_0}$ and $V_{0}$ respectively
when $g_0 \neq 0$ or $\{z, u_{g_0,1},...,
u_{g_0,n_{g_0}},\hat{u}_{g_0,1},..., \hat{u}_{g_0,n_{g_0}}\}$ when
$g_0=0$. For any subset $\{g, -g+g_0\} \neq \{g_0,0\}$ of $G$, fix
also basis $\{u_{g,1},..., u_{g,n_g}\}$ and $\{\hat{u}_{g,1},...,
\hat{u}_{g,n_g}\}$ of $V_g$ and $V_{-g+g_0}$ in case $2g \neq g_0$
and $\{u_{g,1},..., u_{g,n_g}\}$ basis of $V_g$ in case  $2g =
g_0.$ Then  by defining a product on $V$ given by
$[{u_{g,i}},\hat{u}_{g,i}]=z$,
$[\hat{u}_{g,i},{u_{g,i}}]=-\epsilon(-g+g_0,g) z$ in the cases
$g=0$ or $2g \neq g_0$, $[{u}_{g,i},{u_{g,i}}]=z$ in the cases
 $2g = g_0$ with $g \neq 0$, and the remaining brackets zero,  for any
subset $\{g, -g+g_0\}$ of $G$, we get that $V$ becomes a
Heisenberg Lie color algebra that we call of {\it type $(G,g_0,
\epsilon)$}. We note that for an easier notation we allow empty
basis   in the above definition which correspond to trivial
subspaces $\{0\}$.

\medskip

\begin{lemma}\label{distintos} A Heisenberg Lie color algebra $L$ of type $(G,g_0, \epsilon)$    is a graded Heisenberg superalgebra if and only if  $\epsilon(g,-g+g_0)
\in \{\pm1\}$ for any  $g \in G$ such that $L_g \neq 0$.
 \end{lemma}
\begin{proof}  Suppose  $L$ is  a grading
of a Heisenberg superalgebra $L=L_0 \oplus L_1$  and there exists
$g \in G$ with $L_g \neq 0$ and such that  $\epsilon(g,-g+g_0)
\notin \{\pm1\}$. Since $[L_0,L_1]=0$,  $L_{g}+L_{-g+g_0}\subset
L_i$ for some $i \in {\mathbb Z}_2.$ By taking either $u_g \in
L_{g}$ and $\hat{u}_g \in L_{-g+g_0}$ in the cases $g=0$ or  $2g
\neq g_0$;  or $u_g \in L_{g}$ in the case $2g = g_0$ with $g \neq
0$, elements of the standard basis of $(G,g_0, \epsilon)$
described in Example~3, we have either $0 \neq [u_g,\hat{u}_g
]=-\epsilon(g, -g+g_0)[\hat{u}_g,u_g ]$ if $2g \neq g_0$ or $0
\neq [u_g,{u_g} ]=-\epsilon(g, -g+g_0)[{u_g},u_g ]$ if $2g =g_0$,
being $\epsilon(g, -g+g_0) \neq \pm 1$, which contradicts the
identities of a Lie superalgebra.

Conversely, if  $\epsilon(g,-g+g_0) \in \{\pm1\}$ for any $g \in
G$ with $L_g \neq 0$, we can ${\mathbb Z}_2$-graduate   $L$ as
$$L=(\bigoplus\limits_{\{g \in {\rm supp}(G): \epsilon(g,-g+g_0)=1\}}L_g) \oplus
(\bigoplus\limits_{\{h \in {\rm supp}(G):
\epsilon(h,-h+g_0)=-1\}}L_h),$$ this one becoming a Heisenberg Lie
superalgebra, graded by the group generated by the support ${\rm
supp}(G):=\{g\in G:L_g\ne0\}$.
\end{proof}


\begin{pr}\label{color}
Any Heisenberg Lie color algebra  is  isomorphic to  a  Heisenberg
Lie color algebra of type $(G,g_0, \epsilon)$.
\end{pr}
\begin{proof}
 Consider  a Heisenberg Lie
color algebra $L= \bigoplus\limits_{g \in G}L_g$.  
We rule out the trivial case $G=\{0\}$ which gives the Heisenberg Lie color algebra of type $(\{0\},0,1)$, that is,  a Heisenberg algebra with trivial grading.
Since $\dim({\mathcal Z}(L))=1$, we can write ${\mathcal Z}(L)=\langle z
\rangle$, being $z=\sum_{i=1}^n x_{g_i}$, with any $0 \neq x_{g_i}
\in L_{g_i}$ and $g_i \neq g_j$ if $i \neq j$. If $n\ne1$, then
$[x_{g_i}, L_h]\subset L_{g_i+h}\cap\span{z}=0$ for any $h \in G$
and $i \in \{1,..,n\}$. Hence any $x_{g_i} \in {\mathcal
Z}(L)=\span{z}$, a contradiction. Thus $z \in L_{g_0}$ for some
$g_0 \in G$.  Now the fact $[L,L]\subset{\mathcal Z}(L)$ gives us
that for any $g \in G$,   $[L_g, L_h]=0$ if $g+h \neq g_0$, and
consequently  $[L_g, L_{-g+g_0}]\neq 0$ if $g \neq g_0$.

Since for any skew-symmetric bicharacter $\epsilon\colon G \times
G \longrightarrow {\hu F}^{\times}$ and $g \in G$ we have
$\epsilon(g,g)\in \{\pm 1\}$, we can ${\mathbb Z}_2$-graduate  $L$
as
$$
L= (\bigoplus\limits_{\{g \in G: \epsilon(g,g)=1\}}L_g ) \oplus
 (\bigoplus\limits_{\{h \in G: \epsilon(h,h)=-1\}}L_h ).
$$
This is a grading on the algebra $L$ since
$\epsilon(g+h,g+h)=\epsilon(g,g)\epsilon(g,h)\epsilon(h,g)\epsilon(h,h)=
\epsilon(g,g)\epsilon(h,h)$.

Let us distinguish two cases, according to the dichotomy $g_0=0$ or $g_0\ne 0$.

 First, assume $g_0=0$.
 For any $g \in G$ it is easy to check that
$\epsilon(g,g)=\epsilon(-g,-g)=\epsilon(g,-g)=\epsilon(-g,g) \in
\{\pm 1\}.$
This fact together  with the observations in the previous
paragraph tell us  that $L=(\bigoplus\limits_{\{g \in G:
\epsilon(g,g)=1\}}L_g) \oplus (\bigoplus\limits_{\{h \in G:
\epsilon(h,h)=-1\}}L_h)$ is actually a   Lie superalgebra,
satisfying $[L,L]={\mathcal Z}(L)$ and $\dim ({\mathcal Z}(L))=1$,
and  being the initial  Lie color grading a refinement of the
${\mathbb Z}_2$-grading as superalgebra. By arguing as in
\cite{Camacho} we easily get that  $L$ is of type $(G,0,
\epsilon)$, being also a grading of a Heisenberg superalgebra.

Second, assume that $g_0 \neq 0$. Since
$[L,L]\subset\span{z}\subset L_{g_0}$, then
$[L_{g_0},L_{g_0}]\subset L_{2g_0}\cap L_{g_0}=0$, so that
 $L^{\prime}:=L_{g_0} \oplus L_0$ is a  Lie algebra (take into consideration $\epsilon(g_0,0)=\epsilon(0,0)=1$).
If $L_0=0$, then  $L^{\prime}=\langle z \rangle$ and otherwise
$[L^{\prime},L^{\prime}]={\mathcal Z}(L^{\prime})$ with $\dim
({\mathcal Z}(L^{\prime}))=1$. In the second situation, we have
that $L^{\prime}=L_{g_0} \oplus L_0$ is a Heisenberg algebra, so
that taking into account Section~\ref{Sec_normalesHei}, the
grading is toral and there exist basis $\{z, u_{g_0,1},...,
u_{g_0,n_{g_0}}\}$ and $\{\hat{u}_{g_0,1},...,
\hat{u}_{g_0,n_{g_0}}\}$ of $L_{g_0}$ and $L_{0}$ respectively
such that $[{u_{g_0,i}},\hat{u}_{g_0,i}]=z$,
$[\hat{u}_{g_0,i},{u_{g_0,i}}]=-\ z$ and such that the remaining
products in $L^{\prime}$ are zero.
Consider now any  subset $\{g, -g+g_0\} \neq \{g_0,0\}$ of $G$. In
case $L_g \neq 0$, then necessarily $L_{-g+g_0} \neq 0$ and we can
distinguish two possibilities. First, if $2g \neq g_0$, the facts
$[L_g, L_h]=0$ if $g+h \neq g_0$ and $[L_g, L_{-g+g_0}]=\langle z
\rangle$ allow us to apply  standard linear algebra arguments to
obtain  basis $\{u_{g,1},..., u_{g,n_g}\}$ and
$\{\hat{u}_{g,1},..., \hat{u}_{g,n_g}\}$ of $L_g$ and $L_{-g+g_0}$
such that $[{u_{g,i}},\hat{u}_{g,i}]=z$,
$[\hat{u}_{g,i},{u_{g,i}}]=-\epsilon(-g+g_0,g) z$ and being null
the rest of the products among the elements of the basis. Second,
in the case $2g = g_0$, a similar   argument gives us
$\{u_{g,1},..., u_{g,n_g}\}$ a basis of $L_g$ such that
$[{u}_{g,i},{u_{g,i}}]=z$ with the remaining brackets zero. Also
observe that necessarily $\epsilon(g,g)=-1$ for any $g \in G$ such
that $2g=g_0$ and $L_g \neq 0$ because in the opposite case
$\epsilon(g,g)=1$ and then $0\neq L_g \subset {\mathcal Z}(L)
\subset L_{g_0}$, a contradiction.
 Summarizing, we have showed  that $L$ is
isomorphic to a Heisenberg Lie color algebra of type
$(G,g_0,\epsilon)$ with $g_0 \neq 0$.
\end{proof}

\medskip

We finish this section by showing how the results in Section
\ref{Sec_superHei}  can be applied to the classification of
Heisenberg Lie color algebras.
 Following Lemma \ref{distintos} and the arguments in the proof of
Proposition \ref{color}, any Heisenberg Lie color algebra $L=
\oplus_{g\in G}L_g$  is isomorphic to a grading of a Heisenberg
superalgebra if and only if either  ${\mathcal Z}(L) \subset L_0$
or $L$  is of the type $(G,g_0,\epsilon)$ with $\epsilon(g,-g+g_0)
\in \{\pm1\}$ for any  $g \in G$ such that $L_g \neq 0$. In
particular,  this is the case of the Heisenberg Lie color algebras
considered in \cite{chinos}. Hence, $L$ is isomorphic to a
coarsening of  a  fine grading $H_{n,m}=\oplus_{k\in
K}(H_{n,m})_k$ of a Heisenberg superalgebra.
 Since it
is known the procedure  to compute all of the coarsenigs of a
given grading when this is given by its universal group grading
(see Section~2 and \cite{G2}), we can   apply Theorem~\ref{superteo} to get the list of all of these Heisenberg Lie
color algebras $L$, in the moment $G$ and $K$ are generated by
their supports and $K$ is
 the universal grading group. Of course this procedure does not
 hold for a  Heisenberg  Lie color algebra  which is not a grading of a Heisenberg superalgebra.

\section{Gradings on   twisted Heisenberg  algebras}\label{Sec_torcidasHei}

In this final section we consider the so called twisted
Heisenberg algebras. As mentioned in the introduction, these
algebras appear naturally as some of the direct summands of the
Lie algebras of connected Lie groups acting isometrically and
locally faithfully on compact connected Lorentzian manifolds.

The ground field $\mathbb{F}$ will be assumed to be algebraically closed of characteristic zero.

\subsection{Geometric definition of twisted Heisenberg algebra.}

We follow the approach of \cite[Chapter~8]{Adams}, adapted to our context.
Take $\lambda:=(\lambda_1,\ldots,\lambda_k)\in\mathbb
(\mathbb R^\times)^k$. For each $t\in\mathbb{R}$, take $\theta_t\colon H_n\to H_n$ the automorphism of the Lie algebra given in terms of the basis in Equation~(\ref{base1}) by
$$
\begin{array}{rcl}
\theta_t(e_j)&=&\cos(\lambda_jt)\,e_j+\sin(\lambda_jt)\,\hat e_j,\\
\theta_t(\hat e_j)&=&-\sin(\lambda_jt)\,e_j+\cos(\lambda_jt)\,\hat e_j,\\
\theta_t(z)&=&z.
\end{array} 
$$
Thus we have $\{\theta_t: t\in\mathbb{R}\}$ a uniparametric subgroup of $\aut(H_n)$. This induces, by monodromy, an action of $\mathbb{R}$ by Lie group automorphisms of $\mathfrak{H}_n$, the Lie group of $H_n$.  The group $\mathfrak{H}_n^\lambda:=\mathbb{R}\ltimes \mathfrak{H}_n$  is then called a \emph{twisted Heisenberg group}, and its tangent Lie algebra is called a \emph{twisted Heisenberg Lie algebra}, and it is denoted by $H_n^\lambda$.
According to \cite[Lemma 8.2.1]{Adams}, $H_n^\lambda$ has a basis $\{X_1,\dots, X_k,Y_1,\dots,Y_k,W,Z\}$ such that
$[X_i,Y_i]=Z$ (a central element), $[W,X_i]=\lambda_iY_i$ and $[W,Y_i]=-\lambda_iX_i$.

\subsection{Algebraic definition of twisted Heisenberg algebra.}
Consider the Heisenberg algebra $H_n$ over our algebraically closed field $\mathbb F$
and take $d$ to be any derivation of $H_n$. Then one can define in
$\mathbb F\times H_n$ the product
$$[(\lambda,a),(\mu,b)]:=(0,\lambda d(b)- \mu d(a)+[a,b]),
$$
for any $\lambda,\mu\in\mathbb{R}$, $a,b\in H_n$.
This defines a Lie algebra structure in $\mathbb F\times H_n$ and
we will denote this Lie algebra by $H_n^d$. If we
define $u=(1,0)$, then $H_n^d={\mathbb F} u\oplus H_n$ and its
product can be rewritten as $[\lambda u+a,\mu
u+b]=\lambda[u,b]-\mu[u,a]+[a,b]$. Thus,
$d=\ad(u)\vert_{H_n}$.
If $d_1$ and $d_2$ are   derivations of $H_n$ such that $H_n=\textrm{Im}(d_i)+F z$, it is easy to see that
$H_n^{d_1}\cong H_n^{d_2}$ if and only if there is some nonzero scalar
$\lambda_0$, an element $x_0\in H_n$, and  an automorphism $\theta\in\aut(H_n)$ such that
$\theta(\lambda_0 d_1+\hbox{ad}(x_0))\theta^{-1}=d_2$. In particular if the two derivations $d_1$ and $d_2$ are in the same orbit on the action
of $\aut(H_n)$ on $\text{Der}(H_n)$ by conjugation, then both algebras $H_n^{d_1}$ and $H_n^{d_2}$ are isomorphic.
\medskip

There is a more intrinsic definition of this kind of
algebras which is equivalent to that of $H_n^d$. On   one hand,
the Lie algebra $H_n^d$ fits in a split exact sequence
$$
0\to H_n\buildrel{i}\over{\to} H_n^d\buildrel{p}\over{\to}\mathbb
F\to 0
$$
where $i$ is the inclusion map and $p(\lambda,a)=\lambda$ for any
$a\in H_n$. The sequence is split because we can define $j\colon
{\mathbb F}\to H_n^d$ by $j(1)=(1,0)$ and then $pj=1_{\mathbb F}$.
On the other hand if we consider any algebra $A$ which is a split
extension of the type
$$
0\to H_n\to A\buildrel p\over\to\mathbb F\to 0,
$$
then $A$ is isomorphic to some $H_n^d$ for a suitable derivation
$d$ of $H_n$.

Observe that $H_n^d$ is isomorphic to $H_n^\lambda$  if one takes the  particular
derivation of $H_n$ given in terms of the basis $\{z,e_1,
\hat{e}_1,...,e_k, \hat{e}_k\}$ of $H_n$ ($n=2k+1$) by $d(z)=0$,
$d(e_i)=\lambda_i \hat{e}_i$ and $d(\hat{e}_i)=-\lambda_i e_i$ for
a fixed $k$-tuple $\lambda:=(\lambda_1,\ldots,\lambda_k)\in\mathbb
(\mathbb Q^\times)^k\subset(\mathbb F^\times)^k$, since $[u,e_i]=\lambda_i\hat{e}_i$ while $[u,\hat{e}_i]=-\lambda_i e_i$.
This motivates that,  for any choice of $\lambda\in(\mathbb F^\times)^k$, we denote such an algebra by $H_n^\lambda$ and we call it also
a \emph{twisted Heisenberg Lie algebra}, although algebraically it is certain extension of $H_n$ more than a twist.
A definition depending on a   basis is convenient because of our necessity of
making explicit computations when dealing with fine gradings.
Besides note   that under a suitable change of coordinates, the basis 
$\{z,u,e_1,
\hat{e}_1,...,e_k, \hat{e}_k\}$ can be chosen to satisfy
\begin{equation}\label{eq_labasebuenadelatwisted}
[e_i, \hat{e}_i]  = \lambda_iz\,, \quad [u,e_i] =
\lambda_i\hat{e}_i \ \text{ and } \ [u, \hat{e}_i] =
\lambda_ie_i,\,
\end{equation}
where we have used $\I=\sqrt{-1}\in \mathbb{F}$.
Thus, to fix the ideas we give the following:

\begin{de}
Let $\lambda = (\lambda_1,\dots,\lambda_k) \in
(\mathbb{F}^{\times})^k$, $k > 0$. The corresponding {\it twisted
Heisenberg algebra} $H_n^\lambda$ of dimension $n = 2k + 2$ is
the Lie algebra spanned by the elements
\begin{equation} \label{exbasetw}
   \{z,u,e_1,
\hat{e}_1,...,e_k, \hat{e}_k\},
\end{equation}
 and the nonvanishing Lie brackets
are given by Equation~(\ref{eq_labasebuenadelatwisted})
for any $i=1, \dots, k.$
\end{de}

As before, $H_n^\lambda = \mathbb{F}u\oplus H_n$ where $H_n$ can
be identified with the subalgebra spanned by
$\{z,e_1,\hat{e}_1,..., e_k,\hat{e}_k\}$.
This algebra is not nilpotent, but it is solvable, since
$[H_n^\lambda,H_n^\lambda]= H_n$.
\smallskip

The isomorphism condition given above in terms of derivations can be applied to $H_n^\lambda$
and so a direct argument proves that given $\lambda=(\lambda_i)\in ({\mathbb F}^\times)^n$ and
$\mu=(\mu_i)\in ({\mathbb F}^\times)^n$,
one has $H_n^{\lambda}\cong H_n^{\mu}$ if and only if there is a permutation $\sigma\in S_n$ and a scalar $k\in{\mathbb F}^\times$ such that
$\mu_i=k\lambda_{\sigma(i)}$ for all $i\le k$. 
\smallskip

\subsection{Torality and basic examples}\label{subsec_casofaciltwisted}

We are now dealing with two fine gradings on $H_n^\lambda$ which
will be relevant to our work. One of them is toral while the other
is not.
\medskip

A fine grading on $H_n^\lambda$ is obviously provided by our basis
in Equation~(\ref{exbasetw}):
$$H_n^\lambda = \langle z\rangle \oplus \langle u \rangle \oplus
(\oplus_{i=1}^k \langle e_i \rangle ) \oplus (\oplus_{i=1}^k
\langle \hat{e}_i\rangle).$$
Again it is also a
group grading. In order to find $G_0$ the universal 
grading group, note that necessarily (we denote  $\deg x = g$ when
$x \in (H_n^\lambda)_g$) the following assertions about the
degrees are verified:
$$
\begin{array}{l}
\deg{e_i} + \deg{\hat{e}_i} = \deg z,\\
\deg{e_i} + \deg{u} = \deg \hat{e}_i,\\
\deg{\hat{e}_i} + \deg{u} = \deg e_i.
\end{array}
$$
Hence $\deg u \in G_0$ has order $1$ or $2$ and $2\deg{e_i} = \deg
z + \deg u$, so that $\deg z$ can be chosen with freedom and
$2(\deg e_i - \deg\hat{e}_i)=0$ (providing $k-1$ more order two
elements). The universal grading group $G_0$ is the   abelian
group with generators $\deg{u}$, $\deg{z}$, $\deg{e_i}$,
$\deg{\hat{e}_i}$ and relations above. It can be computed to be
$G_0 = \mathbb{Z}\times \mathbb{Z}_2 \times \mathbb{Z}_2^{k-1}$
and the grading is given as follows:
$$
\begin{array}{ll}
\Gamma_1:& (H_n^\lambda)_{(2;\bar 1;\bar 0,  \dots,\bar 0,\dots ,\bar 0)}=\langle {z}\rangle ,\\
& (H_n^\lambda)_{(0;\bar 1;\bar 0,  \dots,\bar 0,\dots ,\bar 0)}=\langle {u}\rangle, \\
& (H_n^\lambda)_{(1;\bar 1;\bar 0,  \dots,\bar 1,\dots ,\bar 0)}=\langle {e_i} \rangle\  \text{ ($\bar1$ in the $i$-th slot when $i\ne k$),}\\
& (H_n^\lambda)_{(1;\bar 0;\bar 0,  \dots,\bar 1,\dots ,\bar 0)}=\langle {\hat{e}_i} \rangle,\\
& (H_n^\lambda)_{(1;\bar 1;\bar 0,  \dots,\bar 0,\dots ,\bar 0)}=\langle {e_k} \rangle,\\
& (H_n^\lambda)_{(1;\bar 0;\bar 0,  \dots,\bar 0,\dots ,\bar 0)}=\langle {\hat{e}_k} \rangle.\\
\end{array}
$$
 The  arguments about  torality  in   Section~$2$    imply  that this fine grading is nontoral.

Now we will find   a toral fine grading. First note that, for
\begin{equation}\label{usyuves}
\begin{array}{l}
u_i := e_i + \hat{e}_i,\\
v_i := e_i - \hat{e}_i,
\end{array}
\end{equation}
the following relations are satisfied:

$$\begin{array}{l}
{[u,u_i]} = \lambda_i u_i,\\
{[u,v_i]} = -\lambda_i v_i,\\
{[u_i,v_i]} = -2\lambda_iz.
\end{array}$$

\begin{pr}
\
\begin{itemize}
\item[a)] The group of automorphisms of $H_n^\lambda$ which
diagonalize relative to the basis $B:=\{u,z\}\cup\{u_i,v_i\mid i=1,\dots,k\}$
is a maximal torus. It is given by the linear group whose elements
are the matrices $\mathop{\hbox{\rm diag}}(1,\gamma
,\ldots,\alpha_i,\gamma\alpha_i^{-1},\ldots)$. \item[b)] In any
toral grading $\Gamma$ of $H_n^\lambda$ the identity element of
the grading group is in the support of $\Gamma$.
\end{itemize}
\end{pr}
\begin{proof}
Denote by $T$ the group of automorphisms of $H_n^\lambda$ which
diagonalize relative to the above mentioned basis. Let $f\in T$ and write
$f(u)=\eta u$, $f(z)=\gamma z$, $f(u_i)=\alpha_i u_i$ and
$f(v_i)=\beta_i v_i$ with $\eta,\gamma,\alpha_i,\beta_i\in{\mathbb
F}^\times$. Applying $f$ to $[u,u_i]=\lambda_i u_i$ we get
$\eta\alpha_i=\alpha_i$, hence $\eta=1$. Moreover, since
$[u_i,v_i]=-2\lambda_i z$, again applying $f$ we get $\alpha_i
\beta_i=\gamma$, hence $\beta_i=\gamma\alpha_i^{-1}$. We observe
that $T\cong({\mathbb F}^\times)^{k+1}$ is a torus. To prove its
maximality take an automorphism $g$ commuting with each element in
$T$. Then it must preserve the simultaneous eigenspaces relative
to the elements in $T$. These simultaneous eigenspaces are ${\mathbb F}u$,
${\mathbb F}z$, and all the spaces ${\mathbb F}u_i$ and ${\mathbb
F}v_i$. Thus $g$ must diagonalize $B$ and so $g\in T$.

Finally take a toral grading $\Gamma$ of $H_n^\lambda$ with
grading group $G$. Consider the associated action
$\rho\colon\chi(G)\to\aut(H_n^\lambda)$ where
$\chi(G):=\hom(G,{\mathbb F}^\times)$ is the character group. The
torality of $\Gamma$ implies that $\text{Im}(\rho)$ is contained
in a maximal torus of $\aut(H_n^\lambda)$. Since any two maximal
tori are conjugated, we may assume that $\text{Im}(\rho)\subset
T$. Hence $u$ is fixed by any element in $\text{Im}(\rho)$ and so
it is in the zero homogeneous component $(H_n^\lambda)_0$.
 \end{proof}

 Thus we obtain a toral fine $\mathbb{Z}^{1+k}$-grading
given by
\begin{equation}\label{lagradtoraltorcida}
\begin{array}{ll}
\Gamma_2:&(H_n^\lambda)_{(0;0,\dots,0)}=\langle u\rangle,\\
&(H_n^\lambda)_{(2;0,\dots,0)}=\langle z\rangle,\\
&(H_n^\lambda)_{(1;0,\dots,1,\dots,0)}=\langle u_i\rangle,\\
&(H_n^\lambda)_{(1;0,\dots,-1,\dots,0)}=\langle v_i\rangle,
\end{array}
\end{equation}
if $i=1,\dots,k$.

Observe that the zero homogeneous component of the grading
$\Gamma_1$ is zero, so  applying the last assertion in the
previous proposition we recover the result that this grading is
nontoral.

\medskip

Our first aim is to prove that under suitable conditions on the
vector $\lambda$, $\Gamma_1$ and $\Gamma_2$ are the unique fine
gradings up to equivalence. For that, it is useful to note a more
general fact: in any grading  the element $u$   in Equation~(\ref{exbasetw})
can be taken homogeneous, that is, $u$ can be replaced with a homogeneous element satisfying the same relations.

\begin{lemma}\label{le_uhomogneo}
For each group grading on $H_n^\lambda$, there is a basis $\{z,u',u'_1,v'_1,\dots,u'_k,v'_k\}$  of $H_n^\lambda$ with $u'$ a homogeneous element of the grading,
such that the only nonzero brackets are given by
$$
\begin{array}{l}
{[u',u'_i]} = \lambda_i u'_i,\\
{[u',v'_i]} = -\lambda_i v'_i,\\
{[u'_i,v'_i]} = -2\lambda_iz.
\end{array}
$$
\end{lemma}

\begin{proof}
Let $\Gamma \colon L = \oplus_{g \in G}L_g$ be a group grading
on $L = H_n^\lambda$.  As any automorphism leaves invariant $[L,L]
= H_n$ and $\mathcal{Z}(L) = \langle z\rangle$, this implies that
$z$ is homogeneous   and $H_n$ is a graded subspace (the
homogeneous components of any element in $H_n$ are again elements
in $H_n$). 

Since not every homogeneous element is contained in $H_n$ there must be someone of the form $a=\lambda u+h$ with $\lambda$ a nonzero scalar and
$h\in H_n$. Then $u':=\lambda^{-1}a$ is a  homogeneous element  and $u'-u\in H_n$. Hence $u' = u + \alpha z + \sum_{i=1}^k
\alpha_i u_i + \sum_{i=1}^k \beta_i v_i$ for some choice of scalars
$\alpha,\alpha_i,\beta_i\in\mathbb{F}$. If we take $u_i' = u_i + 2\beta_i
z$ and $v_i' = v_i + 2\alpha_i z$, then the basis $\{z,u',u'_1,v'_1,\dots,u'_k,v'_k\}$
trivially satisfies the required conditions.
\end{proof}

Consequently, fixed $\Gamma \colon L = \oplus_{g \in G}L_g$   a group grading
on $L = H_n^\lambda$, there is not loss of
generality in supposing  that $u$ is homogeneous.
Let us denote by $h \in G$ the degree of $u$ in $\Gamma$. Our next
aim is to show  that  $h$ is necessarily of finite order. From now
on we are going to  denote by $\varphi$ be the inner derivation
$$
\varphi:= \ad(u): H_n^\lambda \to H_n^\lambda,\qquad  x \mapsto
[u,x],
$$
 which is going to be a key tool in the study of  the group
gradings on $H_n^\lambda$. If $0\ne x\in [u,L]$ is a homogeneous
element, then there is $g\in G$ such that $x=\sum_i(c_iu_i+d_iv_i)
\in L_g$ for some scalars $c_i,d_i\in \mathbb{F}$, so that $
\varphi^t(x) = \sum_i(c_iu_i+(-1)^td_iv_i)\lambda_i^t \in
L_{g+th}$ is not zero for all $t \in \mathbb{N}$. Taking into account that at most
there are $2k$ independent elements in the left set
$$
\{\sum_{i=1}^k(c_iu_i+(-1)^td_iv_i)\lambda_i^t : t = 0,1,2,\dots\}
\subset \langle \{u_1,v_1,\dots,u_k,v_k\}\rangle,
$$
hence there is a positive integer $r\le 2k$ with $\varphi^r(x) \in \langle
\{\varphi^t(x)\mid 0\le t<r\}\rangle$.
Since $\varphi^r(x) \in L_{g+rh}
\cap (\sum_{t<r}L_{g+th})$, necessarily there exists  $t < r$ such that $g + rh = g + th$,
 so that $(r-t)h = 0$, as we wanted to show.

Let us denote by $l$ the order of $h$ in $G$.
 Recall that
the set of eigenvalues of $\varphi$ is
$\{\lambda_1,-\lambda_1,\dots,\lambda_k,-\lambda_k,0,0\}$ with
respective eigenvectors $\{u_1,v_1,\dots,u_k,v_k,u,z\}$
(recall Equation~(\ref{usyuves})), so that the set of eigenvalues of
$\varphi\vert_{[u,L]}$ is
$$
\{\lambda_1,-\lambda_1,\dots,\lambda_k,-\lambda_k\}=\Spec(\ad(u)\vert_{[u,L]})
=: \Spec(u).
$$
 Fix some $\lambda_i \in \Spec(u)$, consider the eigenspace of $\varphi$ given by
 $V_{\lambda_i}:=\{x \in L: \varphi (x) =\lambda_i  x\}$ and
   denote by
$$
V_{\lambda_i}^l:=\{x \in L: \varphi^l(x) =\lambda_i^l x\}.
$$
It is obviously nonzero, because $u_i\in V_{\lambda_i} \subset
V_{\lambda_i}^l$. Moreover, as $V_{\lambda_i}^l$ is invariant
under $\varphi$, we have that $\varphi\vert_{V_{\lambda_i}^l}$ is
diagonalizable and
\begin{equation}\label{ccoo}
V_{\lambda_i}^l=\bigoplus\limits_{j=0}^{l-1}V_{\xi^j \lambda_i}
\end{equation}
for $\xi$ a fixed primitive $l$th root of the unit. Note that if $x\in
V_{\lambda_i}^l$, then     $
 \sum_{q=0}^{l-1} (\xi^{ -j}
\lambda_i^{-1})^q\varphi^q(x)\in V_{\xi^j \lambda_i}  $ for any $j=0,\dots, l-1$.

Recall that if $f\in\mathop{\hbox{\rm End}}(L)$ satisfies
$f(L_g)\subset L_g$ for all $g\in G$, then for each $\alpha\in
\mathbb{F}$, the eigenspace $\{x\in L: f(x)=\alpha x\}$ is
graded. This can be applied to
the endomorphism $ \varphi^l$, since $\varphi^l(L_g)\subset L_{g+lh}=L_g$,
so that  the eigenspace $V_{\lambda_i}^l$ {is a graded subspace} of $L$. Thus we
can take $0\ne x\in V_{\lambda_i}^l\cap L_g$ for some $g\in G$.
For each $j=0,\dots, l-1$, the element   $   \sum_{q=0}^{l-1} (\xi^{l-j}
\lambda_i^{-1})^q\varphi^q(x)\in \sum_{q=0}^{l-1}  L_{g+qh}$ must be nonzero,
because the involved homogeneous pieces are different
($g+qh=g+ph$ implies $(q-p)h=0$) and the projection in the component $L_g$ is $x\ne0$.
Consequently $V_{\xi^j \lambda_i}\ne0$ for all $j$, so that
\begin{equation}\label{dos}
\{\xi^j\lambda_i: j=0,...,l-1\} \subset \Spec(u)
\end{equation}
for any $\lambda_i \in \Spec(u)$, and hence
$\Spec(u)=\{\pm\xi^j\lambda_i: 0\le j<l,1\le i\le k\}$.

\begin{pr}\label{pr_loscasosqueestansiempre}
Assume that $\lambda_i/\lambda_j$ is not a   root of unit   if $i\ne j$.
Then the unique fine (group) gradings on $H_n^\lambda$ (up to
equivalence) are $\Gamma_1$ and $\Gamma_2$. Moreover, the Weyl
groups of these fine gradings are $\mathcal{W}(\Gamma_1)\cong
\mathbb{Z}_2^k$ and $\mathcal{W}(\Gamma_2)\cong\mathbb{Z}_2$.
\end{pr}

\begin{proof}
Let $\Gamma \colon L = \oplus_{g \in G}L_g$ be a grading on $L =
H_n^\lambda$. By the above discussion we can suppose that $u$ is
homogeneous with degree $h \in G$ of finite order $l$. Let us show
that either $l=1$ or $l=2$. Otherwise, take $\xi$ a primitive
$l$th root of unit. As $\xi\lambda_1\ne\pm\lambda_1$ belongs to $\Spec(u)$ according to
Equation~(\ref{dos}), then there is $1\ne i\le k$ such that
$\xi\lambda_1\in\{\pm\lambda_i\}$ and hence either
$\left(\frac{\lambda_i}{\lambda_1}\right)^l=1$ or
$\left(\frac{\lambda_i}{\lambda_1}\right)^{2l}=1$, what is a
contradiction.
%
%
Hence, we can distinguish two cases.

First consider   $h = 0\in G$. Thus  {$\varphi(L_g) \subset L_g$
for any $g$}.  Restrict $\varphi\colon [u,L]\to [u,L]$. We can
take a basis of homogeneous elements which are eigenvectors for
$\varphi$. Recall that the spectrum of $\varphi\vert_{[u,L]}$
consists of $\{\pm \lambda_1,\dots,\pm\lambda_k\}$.
Take $x_1\ne0$ some homogeneous element in $V_{\lambda_1}$. As
$[x_1,V_{-\lambda_1}] \ne 0$, there is some element $y_1 \in
V_{-\lambda_1}$ in the above basis such that $[x_1,y_1]=
-2\lambda_1 z$ (by scaling $y_1$ if necessary). Now $[u,L] = W \oplus {\mathcal Z}_{[u,L]}(W)$ for
$W := \langle x_1,y_1\rangle $, where $W$ as well as its
centralizer  ${\mathcal Z}_{[u,L]}(W)$ are  {graded} and
$\varphi$-invariant. We continue by induction until finding a
basis of homogeneous elements $\{x_1,y_1,\dots,x_k,y_k\}$ of
$[u,L]$ such that $[x_i,y_i] = -2\lambda_i z, [u,x_i] =
\varphi(x_i) = \lambda_i x_i$ and $[u,y_i] = \varphi(y_i) =
-\lambda_i y_i$. Since $$L= \langle z \rangle  \oplus \langle u
\rangle  \oplus [u,L],$$ we have
 that the map $u \mapsto u, z \mapsto z, x_i \mapsto u_i$ and
$y_i \mapsto v_i$ is a Lie algebra isomorphism which applies
$\Gamma$ into a coarsening of $\Gamma_2$. 

Second consider the case when  {$2h = 0$ but $h \ne 0$}. Thus
$\varphi^2$ preserves the grading $\Gamma$ and it is
diagonalizable with  eigenvalues
$\{\lambda_1^2,\lambda_1^2,\dots,\lambda_k^2,\lambda_k^2,0,0\}$
(counting each   one         with multiplicity $1$). Observe that
 $\varphi$
applies
$\{x \in L  : \varphi^2(x) = \lambda_i^2 x\} = V_{\lambda_i}^2$ into
itself. Moreover this set is graded for each $i$, because $\varphi^2$
preserves the grading. For any $0 \ne x_1 \in V_{\lambda_1}^2 \cap
L_g$ a homogeneous element of the grading, $\varphi(x_1)$ is
independent with $x_1$ (otherwise $\varphi(x_1) \in L_g \cap
L_{g+h}$ but $h \ne 0$ and $\varphi(x_1) \ne 0$). Take $y_1 =
\frac1{\lambda_1}\varphi(x_1)$, which verifies $\varphi(y_1) =
\lambda_1x_1$. Since our ground field is algebraically closed, it contains the square roots of all its elements,
so that if
$[x_1,y_1] \ne 0$, we can scale to get $[x_1,y_1] = \lambda_1 z$,
 and
then we can continue because, as before, $[u,L] = W \oplus
{\mathcal Z}_{[u,L]}(W)$ for $W := \langle x_1,y_1\rangle $, where
both $W$ and its centralizer   are graded and $\varphi$-invariant.
The case $[x_1,y_1] = 0$ does not occur under the hypothesis of
the theorem, since $\lambda_i^2 \ne \lambda_j^2$ if $i\ne j$,
so that $\dim V_{\lambda_i}^2=2$ for all $i$. As
$[V_\alpha,V_\beta] = 0$ if $\alpha + \beta \ne 0$, this implies
that there is $y \in V_{\lambda_1}^2$ with $[x_1,y] \ne 0$ but
$V_{\lambda_1}^2 = \langle x_1,y_1\rangle$. To summarize, if $h$
has order $2$ and $\lambda_i/\lambda_j \notin \{\pm1\}$, we find a
basis of homogeneous elements $\{x_1,y_1,\dots,x_k,y_k\}$ of
$[u,L]$ such that $[x_i,y_i] = \lambda_i z, [u,x_i] = \lambda_i
y_i$ and $[u,y_i]=\lambda_i x_i$, so that the map $u \mapsto u, z
\mapsto z, x_i \mapsto e_i$ and $y_i \mapsto \hat{e}_i$ is a Lie
algebra isomorphism which applies $\Gamma$ into a coarsening of
$\Gamma_1$.

In order to compute the Weyl groups of these fine gradings, recall
that any $f\in\aut(L)$ verifies $0\ne f(z)\in\span{z}$.
If besides $f\in\aut(\Gamma_1) $, then $f(u)\in\span{u}$.
Otherwise, there would exist some $i\le k$ such that either
$f(e_i)\in \langle u \rangle$ or $f(\hat{e}_i)\in \langle u
\rangle$, so that $0\ne f(\lambda_iz)=[f(
{e}_i),f(\hat{e}_i)]\in[u,L]\cap\span{z}=0$. Consider   for each
index $i\le k$, the element in $\aut(\Gamma_1)$ defined by
$\mu_i(e_i) =\I \hat{e}_i, \mu_i(\hat{e}_i) =\I e_i, \mu_i(e_j) =
e_j, \mu_i(\hat{e}_j) = \hat{e}_j$  for each $j \ne i$,  $\mu_i(z)
=  z$ and $\mu_i(u) = u$. Note that  if $r=1,...,k$,  {there are
not} any $i,j \leq k$ such
 that $f(e_i)\in \langle
e_r\rangle$ and $f({e}_j)\in \langle \hat{e}_r \rangle$. Hence we
can compose $f$ with some $\mu_i$'s if necessary to obtain that
$f':=\mu_{i_1} \cdots \mu_{i_s} f$ satisfies
$$
f'(e_i) \in \langle {e}_1\rangle \cup \langle {e}_2\rangle \cup
\cdots \cup \langle {e}_k\rangle
$$
for each $i=1,...,k$. Thus, there is $\sigma \in S_k$  such that
   $f'(z)=\mu z$, $f'(u)=\beta u$,   $f'(e_i)=\gamma_i e_{\sigma(i)}$ and $f'(\hat e_i)=\gamma'_i \hat e_{\sigma(i)}$ for any $i=1,...,k$,
 with $\mu, \beta, \gamma_i,\gamma'_i \in {\mathbb
 F}^{\times}$.
 From here, the equality  $f'([u,e_i])=[f'(u),f'(e_i)]$ implies
 \begin{equation}\label{mari1}
\gamma'_i\lambda_i=\beta\gamma_i\lambda_{\sigma(i)},
\end{equation}
  and finally the condition
 $f'([u, \hat{e}_i])=[f'(u),
 f'(\hat{e}_i)]$ allows us to assert
\begin{equation}\label{mari2}
\gamma_i\lambda_i=\beta\gamma'_i\lambda_{\sigma(i)}.
 \end{equation}
 From Equations~(\ref{mari1}) and  (\ref{mari2}) we easily get
$\lambda_{\sigma(i)} \in \pm
 \beta^{-1} \lambda_i$ for any $i=1,...,k$.
 By multiplying, $\Pi_{i=1}^k\lambda_{\sigma(i)}\in\pm\beta^{-k}\Pi_{i=1}^k\lambda_{i}$ so that $\beta^{2k}=1$.
As $\lambda_{\sigma(i)}/\lambda_i$ is not a root of unit if
$\sigma(i)\ne i$, we conclude that $\sigma=\id$, so that
$f'\in\stab(\Gamma_1)$. In other words, since $\mu_i\mu_j=\mu_j\mu_i$,
$$
\mathcal{W}(\Gamma_1)=\{[\mu_{i_1} \dots \mu_{i_s}]: \,1\le
i_1\le\dots\le i_s\le k\}\cong \mathbb{Z}_2^k,
$$
where $[\ ]$ is used for denoting the class of an element of $\aut(\Gamma_1)$ modulo $\stab(\Gamma_1)$.

For the other case, define the automorphism $\mu\in
\aut(\Gamma_2)$ by means of $\mu(u_i) =\I v_i$ and $\mu(v_i) =\I
u_i$ for all $i$,   $\mu(z) =  z$ and $\mu(u) = -u$. Consider
$f\in\aut(\Gamma_2) $, and note that again there is
$\beta\in\mathbb{F}^\times$ such that $f(u)=\beta u $. If $f(u_i)$
is a multiple of either $u_j$ or $v_j$ for some $j$, this clearly
implies that $f(v_i)$ also is, so that there is $\sigma\in S_k$
such that $f(u_i)\in\span{u_{\sigma(i)}}\cup\span{v_{\sigma(i)}}$
for all $i\le k$. As $\beta[u,f(u_i)]=\lambda_if(u_i)$, then
$\lambda_i\in\{\pm\beta\lambda_{\sigma(i)}\}$. Multiplying as in the above case, we get that
$\beta$ is a root of unit, and again we conclude that $\sigma=\id$. By
composing with $\mu$ if necessary, we can assume that
$f(u_1)\in\span{u_1}$, which implies $\beta=1$. If
$f(u_i)\in\span{v_i}$ for some $i$, then $\beta=-1$, which is a
contradiction, so that $f(u_i)\in\span{u_i}$ for all $i$ and $f$
belongs to $\stab(\Gamma_2)$. We have then proved that
$$
\mathcal{W}(\Gamma_2)=\span{[\mu]}\cong \mathbb{Z}_2.
$$
\end{proof}

\subsection{Fine gradings on twisted Heisenberg algebras}

In the general case (possible roots of the unit among the
fractions of $\lambda_i$'s), the situation is much more
{involved}.
 On one hand, a lot of different fine gradings arise, and on the other hand even the Weyl groups of $\Gamma_1$ and $\Gamma_2$ change.  {There is} a lot of symmetry    in the related  twisted Lie algebra, and its fine gradings are also symmetric.
In order to figure out what is happening, we previously  need to
show a couple of key examples.

First,   for $\xi$ a primitive $l$th root of the unit and $\alpha$
a nonzero scalar, we consider the twisted Heisenberg algebra
$H_{2l+2}^\lambda$ corresponding to
$$\lambda=(\lambda_1,\dots,\lambda_l) = (  \xi\alpha, \xi^2\alpha,
\dots, \xi^{l-1}\alpha,\alpha).$$  Thus $[u,u_i] = \xi^i\alpha
u_i, [u,v_i] = -\xi^i\alpha v_i$ and $[u_i,v_i] = -2\xi^i\alpha z$
for $i=1,\dots,l,$ with the definition of $u_i$'s and $v_i$'s as
in Equation~(\ref{usyuves}). Take now
\begin{equation}\label{bloqueprim}
\begin{array}{l}
x_j = \sum_{i=1}^{l}\xi^{ji}u_i,\\
y_j = -\frac1{2l} \sum_{i=1}^{l}(-1)^{j}\xi^{(j-1)i}v_i,
\end{array}
\end{equation}
if $j=1,...,l$. These elements verify $[u,x_j]=\alpha x_{j+1}$ and
$[u,y_j]=\alpha y_{j+1}$ for all $j\le l-1$. Besides $[x_i,y_j] =
{(-1)^{j}}\frac\alpha{l} (\sum_{r=1}^{l} \xi^{r(i+j)})z$ is not
zero if and only if $i+j= l,2l$, and in such a case
$[x_l,y_l]=(-1)^l\alpha z$ and $[x_{i}, y_{l-i}] =(-1)^{l-i}\alpha
z $ for $i=1,...,l-1$. Note that obviously
$\{x_1,y_1\dots,x_{l},y_{l}\}$ is a family of independent vectors
such that $[x_i,x_j]=0=[y_i,y_j]$ for all $i,j$.


Therefore we have a fine grading on $L=H_{2l+2}^\lambda$ over the group
  $$
  G= {\mathbb{Z}}^2 \times  {\mathbb{Z}}_l,
  $$
   given by
\begin{equation}\label{gradtipoI}
\begin{array}{l}
L_{(0,0,\bar1)}=\langle u \rangle,\\
L_{(1,1,\bar0)}=\langle z \rangle,\\
L_{(1,0,\overline{i})}=\langle x_i \rangle,\\
L_{(0,1,\overline{i})}=\langle y_i \rangle,
\end{array}
\end{equation}
for all $i=1,...,l$.

Take $\gamma\in\mathbb{F}$ such that $\gamma^l=(-1)^l$, and
consider the automorphisms $\theta,\vartheta\in\aut(L)$ defined by
\begin{equation}\label{eq_WeylbloquetipoI}
\begin{array}{llll}
\theta(x_i)=x_{i+1},&\theta(y_i)=y_{i-1},&\theta(z)=-z,&\theta(u)=u;\\
\vartheta(x_i)=\gamma^i y_{i},&\vartheta(y_i)=-\gamma^i
x_{i},&\vartheta(z)=z,&\vartheta(u)=\gamma u;
\end{array}
\end{equation}
where the indices are taken modulo $2l$. It is not difficult to
check that the Weyl group of the grading described in
Equation~(\ref{gradtipoI}) is generated by the classes $[\theta]$
and $[\vartheta]$, elements of order $l$ and $2$ respectively
which do not commute, so that the Weyl group is isomorphic to the Dihedral group
$D_l$.

This example motivates the following definition.

\begin{de}
Let $L$ be any Lie algebra, $z \in L$ a fixed element, $u$ an
arbitrary element and $\alpha \in \mathbb{F}^\times$. A set
$B_l^{\textrm{I}}(u,\alpha)$, which will be referred as a block of
type $\textrm{I}$, is given by a family of $2l$ independent
elements in $L$,
$$
B_l^{\textrm{I}}(u,\alpha)=\{x_1,y_1,\dots,x_{l},y_{l}\},
$$
satisfying that the only non-vanishing products among them are the
following:

$$
\begin{array}{ll}
{[u,x_i]} =   \alpha x_{i+1}&\forall i=1,\dots,l-1,\\
{[u,x_{l}]}  =  \alpha x_{1},&\\
{[u,y_i]}  = \alpha y_{i+1}&\forall i=1,\dots,l-1,\\
{[u,y_{l}]}  = (-1)^l\alpha y_{1},&\\
{[x_i,y_{l-i}]}  = (-1)^{l-i}\alpha z&\forall
i=1,\dots,l-1,\\
{[x_l,y_{l}]}  = (-1)^{l}\alpha z.&
\end{array}
$$
\end{de}

\medskip

As a second example, fix   $\zeta$  a primitive $2l$th root of the
unit    and $\alpha$ a nonzero scalar. Consider now the twisted
Heisenberg algebra $H_{2l+2}^\lambda$ corresponding to
$$
\lambda=(\lambda_1,\dots,\lambda_l) = (
\zeta\alpha,\zeta^2\alpha,\dots,\zeta^{l-1}\alpha,-\alpha).
$$
Again $[u,u_i] = \zeta^{i}\alpha u_i, [u,v_i] = -\zeta^{i}\alpha v_i$
and $[u_i,v_i] = -2\zeta^{i}\alpha z$ for $i = 1,\dots,l$.   Take
now
\begin{equation}\label{bloqueseg}
x_j = \frac{\I}{2\sqrt{l}} \sum_{i=1}^{l}(u_i+(-1)^{j-1}
v_i)\zeta^{(j-1)i}
\end{equation}
 for each integer
$j$. Observe that $\{x_1,\dots,x_{2l}\}$ is a family of linearly
independent elements satisfying $ [u, x_{j}]=\alpha x_{j+1}$ for
any $j=1,...,2l-1$ and $[u,x_{2l}]=\alpha x_1$. A direct
computation gives
\begin{equation}\label{cuenta}
[x_i,x_j]=\frac1{2l} \alpha ((-1)^i+(-1)^{j-1})(
\sum_{k=1}^{l}\zeta^{(i+j-1)k}  )z
\end{equation}
for any $i$ and $j$. If $i+j-1=2l$, then $i$ and $j-1$ are either
both odd or both even and $[x_i,x_{2l+1-i}]=(-1)^i \alpha z\ne0$.
Hence,
$$
[x_1,x_{2l}]=-[x_2,x_{2l-1}]=\dots=(-1)^{l-1}[x_l,x_{l+1}].
$$
Again Equation~(\ref{cuenta}) tells us that the remaining brackets
are zero: if $r=i+j-1$ is odd, then $(-1)^i+(-1)^{j-1}=0$, and, if
$r=i+j-1\ne 2l$ is even, then
$\sum_{k=1}^{l}\zeta^{rk}=0$ since $\zeta^2$ is a primitive $l$th root of unit ($\zeta^{rl}=1$with $\zeta^r\ne1$).


We note that this provides a fine grading on $H_{2l+2}^\lambda$
over the group
$$
G= {\mathbb{Z}}  \times {\mathbb{Z}}_{2l},
$$
  given by
$$
\begin{array}{l}
L_{(0, \overline{1})}=\langle u \rangle,\\
L_{(2,\bar1)}=\langle z \rangle,\\
L_{(1,\overline{i})}=\langle x_i \rangle,
\end{array}
$$
for $i=1,...,2l$.

Take $\rho\in\aut(L)$ defined by
\begin{equation}\label{eq_WeylbloquetipoII}
\rho(x_i)=x_{l+i},\quad\rho(z)=(-1)^lz,\quad \rho(u)=u,
\end{equation}
for all $i=1,\dots,2l$ (mod $2l$). This time the Weyl group of the
grading is isomorphic to $\mathbb{Z}_2$, since it is easily proved
to be generated by the class $[ \rho]$.

 This example gives rise to the next  concept.
\begin{de}
Let $L$ be any Lie algebra, $z \in L$ a fixed element, $u\in L$ an
arbitrary element and $\alpha \in \mathbb{F}^\times$. A set
$B_l^{\textrm{II}}(u,\alpha)$, which will be referred as a block
of type $\textrm{II}$, is given by a family of $2l$ independent
elements in $L$,
$$
B_l^{\textrm{II}}(u,\alpha)=\{x_1,\dots,x_{2l}\},
$$
satisfying that the only non-vanishing products among them are the
following:
$$
\begin{array}{ll}
{[u,x_i]} = \alpha x_{i+1}&\forall i=1,\dots,2l\,(\mathop{\rm{mod}} 2l),\\
{[x_i,x_{2l-i+1}]} = (-1)^i\alpha z&\forall i=1,\dots,2l.
\end{array}
$$

\end{de}

\medskip

In fact, all of the   gradings on a twisted Heisenberg algebra
can be described with blocks of types $\textrm{I}$ and
${\textrm{II}}$, according to the following theorem.

\begin{teo}\label{torcidas_lasunicasgradings}
Let $\Gamma$ be a   $G$-grading on $L = H_n^\lambda$. Let
$z\in\mathcal{Z}(L)$. Then there exist $u\in L$,
positive integers $l,s,r $ such that $l(r + 2s) = 2k = n-2$ ($r=0$
when $l$ is odd) and scalars $ \beta_1, \dots, \beta_s,\alpha_1,
\dots, \alpha_r \in \{\pm\lambda_1, \dots, \pm\lambda_k\}$ such
that
\begin{equation}\label{eq_labasedeunafinatwistedcualquiera}
\{z,u\}\cup \big(\bigcup_{j=1}^s B_l^{\textrm{I}}(u,\beta_j)\big) \cup
\big(\bigcup_{i=1}^r B_{\frac l2}^{\textrm{II}}(u,\alpha_i)\big)
\end{equation}
is a basis of homogeneous elements of $\,\Gamma$, being
zero the bracket of any two elements belonging to different
blocks.
\end{teo}

\begin{proof}

Recall that  $z$ is always a homogeneous element. By Lemma~\ref{le_uhomogneo} and the arguments below, we can
assume that $u$ is also homogeneous of degree $h \in G$,
necessarily of finite order. Let $l\in \mathbb{Z}_{\ge0}$ be the
order of $h$.
 Take $\varphi=\ad(u)$ and consider again the subspaces $V_{\lambda_i}$ and $V_{\lambda_i}^l$.
Recall that $V_{\lambda_i}^l$  is a $\varphi$-invariant  graded
subspace
 for all $\lambda_i\in\Spec(u)$.


Let us discuss first the case   that  $l$ is odd. Fix any $0 \neq
x \in V_{\lambda_1}^l \cap L_g$ for some $g \in G$. Since each
$\varphi^i(x) \in L_{g+ih}$, we have that
$$\{x, \varphi(x),..., \varphi^{l-1}(x)\}$$ is a family of linearly
independent elements of $L$. Now observe that
Equation~(\ref{ccoo}), together with the fact that $l$ is odd, says that
$[V_{\lambda_1}^l,V_{\lambda_1}^l]= 0$ and
$[V_{\lambda_1}^l,V_{-\lambda_1}^l]\neq 0$. From here,
\begin{equation}\label{h1}
\hbox{$[\varphi^i(x), \varphi^j(x)]=0$ for any $i,j=0,1,...,l-1$,}
\end{equation}
and we can take a nonzero homogeneous element $0\neq y\in
V_{-\lambda_1}^l \cap L_p$ such that $[x,y]\neq 0$. By scaling if
necessary, we can suppose $[x,y]=\lambda_1 z$, being then $\deg{
z}=g+p$. As above, we also have that $$\{y, \varphi(y),...,
\varphi^{l-1}(y)\}$$ is a family of linearly independent elements
of $L$ satisfying
\begin{equation}\label{h2}
 \hbox{$[\varphi^i(y), \varphi^j(y)]=0$ for any
$i,j=0,1,...,l-1$.}
\end{equation}
Taking into account that $\varphi^i(x) \in L_{g+ih}$ and $\varphi^j(y)
\in L_{p+jh}$,
 we get that if $[\varphi^i(x),\varphi^j(y)]\neq
0$, then $g+p+(i+j)h=\deg{z}=g+p$, which is only possible if $i+j$
is a multiple of $l$. That is, for each $0\le i,j<l$,
\begin{equation}\label{h3}
\hbox{$[\varphi^i(x),\varphi^j(y)]= 0$ if  $i+j\neq 0,l.$}
\end{equation}
 Also
 note that 
\begin{equation}\label{h4}
[\varphi^i(x),\varphi^{l-i}(y)]= (-1)^{l-i}\lambda_1^l[x,y]\neq 0.
\end{equation}
Indeed, take $\xi$ a primitive $l$th root of the unit and write
$x=\sum_{j=0}^{l-1} a_j$ and $y=\sum_{j=0}^{l-1} b_j$ for $a_j\in
V_{\xi^j\lambda_1}$ and $b_j\in V_{-\xi^j\lambda_1}$, taking into
consideration Equation~(\ref{ccoo}). Then
$$
\begin{array}{r}
[\varphi^i(x),\varphi^{l-i}(y)]=\sum_{j,k}[(\xi)^{ji}\lambda_1^ia_j,(-\xi^k)^{l-i}\lambda_1^{l-i}b_k]=\\
=\lambda_1^l(-1)^{l-i}\sum_j(\xi)^{ji}(\xi)^{j(l-i)}  [a_j,b_j] =
(-1)^{l-i}\lambda_1^l[x,y].
\end{array}
$$

\noindent Finally note that  the family $\{x, \varphi(x),...,
\varphi^{l-1}(x),y, \varphi(y),..., \varphi^{l-1}(y)\}$ is
linearly independent. Indeed, in the opposite case some
$\varphi^i(y)=\beta\varphi^j(x)$, $\beta \in {\mathbb
F}^{\times}$,
 because we are dealing with a
 family of  homogeneous
elements, 
and then $[\varphi^i(y),\varphi^{l-j}(y)]=\beta
[\varphi^j(x),\varphi^{l-j}(y)]\neq 0$, what contradicts Equation
(\ref{h2}).

Taking into account Equations~(\ref{h1}), (\ref{h2}), (\ref{h3})
and (\ref{h4}), we have that
$$\{  \frac{\varphi(x)}{\lambda_1}, \frac{\varphi(y)}{\lambda_1},...,\frac{\varphi^{i}(x)}{\lambda_1^{i}},\frac{\varphi^{i}(y)}{\lambda_1^{i}},..., 
\frac{\varphi^{l}(x)}{\lambda_1^{l}}=x,\frac{\varphi^{l}(y)}{\lambda_1^{l}}=(-1)^ly\}$$
is a block $B_l^{\textrm{I}}(u,\lambda_1)$ of type I.

 Now $[u,L] = W \oplus {\mathcal
Z}_{[u,L]}(W)$ for $W := \span{B_l^{\textrm{I}}(u,\lambda_1)} $,
where $W$ as well as its centralizer   are graded and
$\varphi$-invariant. We continue by iterating this process on
${\mathcal Z}_{[u,L]}(W)$ until finding a basis of $[u,L]$
 formed by $s=\frac kl$ blocks of type I of homogeneous elements.\medskip

 Now    consider  the case with $l$  even. If we think as above of the
 linear subspace
 $0\neq V_{\lambda_1}^l,$ we have two different cases to distinguish.

 Assume first that for any $g \in G$ and any $x \in
 V_{\lambda_1}^l\cap L_g$ we have $[x, \varphi (x)]= 0.$ Fix  $0\neq x \in
 V_{\lambda_1}^l\cap L_g$ for some $g \in G$, being then  $\{x, \varphi(x),..., \varphi^{l-1}(x)\}$  a family of linearly
independent elements of $L$. By induction on $n$ it is easy to
verify, taking into account that $\varphi$ is a derivation,  that
for any $i=0,...,l-1$, we have $[\varphi^i(x),\varphi^{i+n}(x)]=0$
for any $n=1,...,l$. That is, $[\varphi^i(x),\varphi^{j}(x)]=0$
for any $i,j=0,...,l-1.$

Since  the fact that $l$ is even implies
$V_{\lambda_1}^l=V_{-\lambda_1}^l$, we can choose a homogeneous
element $0\neq y \in V_{\lambda_1}^l \cap L_p$, for some $p \in
G$,  such that $0\neq [x,y]=\lambda_1 z.$
The same arguments that in the odd case say that again
$$
\{ \frac{\varphi(x)}{\lambda_1},
\frac{\varphi(y)}{\lambda_1},...,\frac{\varphi^{i}(x)}{\lambda_1^{i}},\frac{\varphi^{i}(y)}{\lambda_1^{i}},...,
\frac{\varphi^{l}(x)}{\lambda_1^{l}},\frac{\varphi^{l}(y)}{\lambda_1^{l}}\}
$$
is a block $B_l^{\textrm{I}}(u,\lambda_1)$ of type I. Now we can
write
\begin{equation}\label{W2}
[u,L] = W \oplus {\mathcal Z}_{[u,L]}(W)
\end{equation}
 for $W :=\span{
B_l^{\textrm{I}}(u,\lambda_1)} $, where $W$ as well as its
centralizer
  are graded and $\varphi$-invariant.

 Second, assume that there exist $g \in G$ and $0\neq x \in
 V_{\lambda_1}^l\cap L_g$ such that $[x, \varphi (x)]\neq 0$. By
 scaling if necessary, we can {assume} $[x, \varphi (x)]= \lambda_1^2 z.$
We have as above that $\{x, \varphi(x),..., \varphi^{l-1}(x)\}$ is
a family of homogeneous linearly independent elements of $L$
satisfying $[\varphi^i(x),\varphi^j(x)]= 0$ if  $i+j\neq 1,l+1$
(take into account that $\deg z=2g+h$).
Besides Equation~(\ref{ccoo}) allows us to get in a
direct way that
$$
[\varphi^i(x), \varphi^{l-i+1}(x)]=(-1)^i\lambda_1^l[x,
\varphi(x)]\neq 0
$$
for any $i=1,\dots,l$.
 Thus the set
$$\{  \frac{\varphi(x)}{\lambda_1}, ...,\frac{\varphi^{i}(x)}{\lambda_1^{i}},..., \frac{\varphi^{l}(x)}{\lambda_1^{l}}=x\}$$
is a block $B_{\frac l2}^{\textrm{II}}(u,\lambda_1)$ of type II.
Now we can write
\begin{equation}\label{W1}
[u,L] = W \oplus {\mathcal Z}_{[u,L]}(W)
\end{equation}
 for $W$ the vector space spanned by the above block $
B_{\frac l2}^{\textrm{II}}(u,\lambda_1) $, where $W$ as well as
its centralizer ${\mathcal Z}_{[u,L]}(W)$ are graded and
$\varphi$-invariant.

Taking into account Equations (\ref{W2}) and (\ref{W1}), we can
iterate this process on ${\mathcal Z}_{[u,L]}(W)$ until finding
the required  basis of $[u,L]$
 formed by $s$ blocks of type I and  $r$ blocks of type~II.
\end{proof}

Theorem~\ref{torcidas_lasunicasgradings} provides, in particular, a description of all the fine gradings on $H_n^\lambda$. 
It is clear that each basis as the one in Equation~(\ref{eq_labasedeunafinatwistedcualquiera}) determines a fine grading, which we will denote by
\begin{equation}\label{eq_nombredelasfinas}
\Gamma(l,s,r;\beta_1,\dots,\beta_s,\alpha_1,\dots,\alpha_r).
\end{equation}
For instance, our well known fine gradings of Subsection~\ref{subsec_casofaciltwisted} are
$$\begin{array}{l}
\Gamma_1=\Gamma(2,0,k;\lambda_1,\dots,\lambda_k),\\
\Gamma_2=\Gamma(1,k,0;\lambda_1,\dots,\lambda_k),
\end{array}
$$
 taking into consideration that
$\{u_i,\frac12v_i\}=B_{1}^{\textrm{I}}(u,\lambda_i)$ and that $\{\hat e_i,
e_i\}=B_{1}^{\textrm{II}}(u,\lambda_i)$.

The point is that the grading   in Equation~(\ref{eq_nombredelasfinas})
does not necessarily exist for all choice of integers $l,s,r$ such that $l(r+2s)=n-2$ and all nonzero scalars $\beta_i,\alpha_j\in\Spec(u)$.
If we are in the situation of Theorem~\ref{torcidas_lasunicasgradings},
and take $\xi$ a primitive $l$th root of
the unit, then
\begin{equation}\label{comodebeserespectro}
\Spec(u) =\{\pm\lambda_i: i=1,\dots, k\}=
\end{equation}
$$
\{\xi^t\beta_j,-\xi^t\beta_j  : j=1, \dots, s, t=0, \dots,
l-1\} \cup\{ \xi^t\alpha_i  : i=1, \dots, r, t=0, \dots, l-1\}.
$$


This condition is not only necessary, but sufficient.


\begin{teo}\label{torcidas_deverdadsongradings} 
A   grading $\Gamma$ on $H_n^\lambda$ is fine if and only if $\Gamma$ is isomorphic to $$\Gamma(l,s,r;\beta_1,\dots,\beta_s,\alpha_1,\dots,\alpha_r)$$
for some  $l,s,r\in\mathbb{Z}_{\ge0}$  such that
$l(r + 2s) = 2k = n-2\, (\ne0)$ and some scalars
$\beta_1,\dots,\beta_s,\alpha_1,\dots,\alpha_r \in
\{\pm\lambda_1,\dots,\pm\lambda_k\}$   such that
(\ref{comodebeserespectro})  {holds}, with $l$ even if $r\ne0$.
The universal grading group in this case is
$$
\begin{array}{ll}
\mathbb{Z}_{l}\times\mathbb{Z}^{s+1} \times \mathbb{Z}_2^{r-1}  & \text{ if }r\ne0,\\
\mathbb{Z}_{l}\times\mathbb{Z}^{s+1}  & \text{ if } r = 0.
\end{array}
$$
\end{teo}

\begin{proof} The fact that any fine grading is like the above has been proved in  Theorem~\ref{torcidas_lasunicasgradings}. For the converse, reorder
$$
\lambda=(\xi\beta_1,\dots,\xi^l\beta_1,\dots, \xi\beta_s,\dots,\xi^l\beta_s,\zeta\alpha_1,\dots,\zeta^l\alpha_1,\dots,\zeta\alpha_r,\dots,\zeta^l\alpha_r),
$$
(with $\zeta^2=\xi$), what is possible because of (\ref{comodebeserespectro}).
Take $\{z,u,u_1,\dots,u_k, v_1,\dots,v_k\}$ a basis as in Equation~(\ref{usyuves}).
Take $\{x_1^j,y_1^j,\dots,x_l^j,y_l^j\}$ a block
$B_l^{\textrm{I}}(u,\beta_j)$ for each $j\le s$, chosen as in Equation~(\ref{bloqueprim}).
Take $\{a_1^t,\dots,a_l^t\}$ a block $B_{\frac
l2}^{\textrm{II}}(u,\alpha_t)$ for each $t\le r$, chosen as in Equation~(\ref{bloqueseg}).
Now the union of these blocks as in Equation~(\ref{eq_labasedeunafinatwistedcualquiera}) provides a basis which determines a fine grading (since the product of two elements in the basis is   multiple of another one). This grading
has universal  group
$\mathbb{Z}_l\times\mathbb{Z}^s\times\mathbb{Z}\times\mathbb{Z}_2^{r-1}$
and it is given by
$$
\begin{array}{l}
\deg(x_i^j)=(\overline{i+1};\,0,\dots,1,\dots,0;\,1;\,\bar0,\dots,\bar0)
\quad\text{($1$ in the $j$th slot)},\\
\deg(y_i^j)=(\overline{i};\,0,\dots,-1,\dots,0;\,1;\,\bar0,\dots,\bar0),\\
\deg(z)=(\overline{1};\,0,\dots ,0;\,2;\,\bar0,\dots,\bar0),\\
\deg(u)=(\overline{1};\,0,\dots ,0;\,0;\,\bar0,\dots,\bar0),\\
\deg(a_i^t)=(\overline{i};\,0,\dots ,0;\,1;\,\bar0,\dots,\bar1,\dots,\bar0)\quad\text{($\bar1$ in the $t$th slot if $t< r$)},\\
\deg(a_i^r)=(\overline{i};\,0,\dots ,0;\,1;\,\bar0,\dots,\bar0).
\end{array}
$$
\end{proof}

In practice, when one wants to know how many gradings are in a
particular twisted Heisenberg  algebra $H_n^\lambda$, it is
enough to see how many ways are of   {splitting}
$\{\pm\lambda_1,\dots,\pm\lambda_k\}$     in the way described in
Equation~(\ref{comodebeserespectro}).
\smallskip

 \begin{example}\rm
 Let us compute how many fine gradings  {are there} in $L=H_{10}^{(1,1,\I,\I)}$.
 As $l$ must divide $8$, the possibilities are $l=1$, $l=2$ with $(r,s)=(4,0),(2,1),(0,2)$ and $l=4$ with $(r,s)=(2,0),(0,1)$.
 But we have  seven  (not six) fine gradings:
 \begin{itemize}
 \item the  $\mathbb{Z}^5$-grading $\Gamma(1,4,0;1,1,\I,\I)$ (the only toral fine grading, $\Gamma_2$);
 \item the  $\mathbb{Z}\times \mathbb{Z}_2^4$-grading $\Gamma(2,0,4;1,1,\I,\I)$ (again  $\Gamma_1$);
 \item two $\mathbb{Z}^2\times\mathbb{Z}_2^2$-gradings: $\Gamma(2,1,2;\I,1,1)$ and $\Gamma(2,1,2;1,\I,\I)$;
 \item the $\mathbb{Z}^3\times\mathbb{Z}_2$-grading $\Gamma(2,2,0;1, \I)$;
 \item the $\mathbb{Z}\times\mathbb{Z}_4\times\mathbb{Z}_2$-grading $\Gamma(4,0,2;1,1)$;
 \item the $\mathbb{Z}^2\times \mathbb{Z}_4$-grading $\Gamma(4,1,0;1)$.
 \end{itemize}
 The possibility $l=8$ does not happen taking into account that $\pm\frac{\lambda_i}{\lambda_j}$ is never a primitive eighth  root of the unit when $\lambda_i,\lambda_j\in\Spec(u)$.
 \end{example}
 \smallskip


For boarding the problem of the classification of the  fine gradings up to equivalence, we need to make some considerations.

\begin{lemma}\label{le_cambioenlosbloques}
Let $L$ be a Lie algebra, $z\in \mathcal{Z}(L)$, $u\in L$,
$\alpha,\beta\in \mathbb{F}^{\times}$ such that $L$ contains
blocks of types $B_l^\nu(u,\alpha)$ and $B_l^\nu(u,\beta)$ for
some $\nu\in\{{\textrm{I,II}}\}$. Then
\begin{itemize}
\item[i)]
$\span{B_l^\textrm{I}(u,\alpha)}=\span{B_l^\textrm{I}(u,\beta)}$
if and only if $\left(\frac\alpha\beta\right)^l=1$ if $l$ is even
and $\left(\frac\alpha\beta\right)^{2l}=1$ if $l$ is odd.
\item[ii)]
$\span{B_l^\textrm{II}(u,\alpha)}=\span{B_l^\textrm{II}(u,\beta)}$
if and only if $\left(\frac\alpha\beta\right)^l=1$.
\end{itemize}
\end{lemma}

\begin{proof}
As usual, denote $\varphi=\ad(u)$ and $\xi$ a primitive $l$th root
of the unit.
For i), take $ B_l^\textrm{I}(u,\alpha)=\{x_1,y_1,\dots,x_l,y_l\}$
and $V$ the vector space spanned by these elements. Note that
$\varphi^l$ diagonalizes $V$ with the only eigenvalues $\alpha^l$ and
$(-1)^l\alpha^l$ (eigenvectors $x_i$'s and $y_i$'s respectively).
Thus either $\alpha^l=\beta^l$ or $\alpha^l=(-1)^l\beta^l$.

Conversely, let us see that there is a block of type $
B_l^\textrm{I}(u,\xi\alpha)$ contained in
$\span{B_l^\textrm{I}(u,\alpha)}$. Indeed, take $\zeta$ such that
$\zeta^2=\xi$. The elements $x_i':=\zeta^{1-2i}x_i$ and
$y_i':=\zeta^{1-2i}y_i$ constitute a block of type $B_l^\textrm{I}(u,\zeta^2\alpha)$. Moreover,
if we take $\delta$ such that $\delta^2=\zeta$, then the   elements
$x_i':=(-1)^i\delta^{1-2i}y_i$ and $y_i':=(-1)^i\delta^{1-2i}x_i$
constitute a block of type $ B_l^\textrm{I}(u, \zeta\alpha)$ if
$l$ is odd.

The case ii) is  {proved} with similar arguments.
\end{proof}

These arguments make convenient to consider the following equivalence
relation  in $\mathbb{F}^{\times}$.
 Two nonzero scalars $\alpha$ and  $\beta$  are $m$-related
  if and only if $(\frac\alpha\beta)^m=1$.
  The
equivalence classes of the element $\alpha$ for the $l$-relation and the $2l$-relation will be denoted respectively by
$$
\begin{array}{l}
\widehat\alpha:=\{\alpha\xi^t: t=0,\dots,l-1\},\\
\widetilde\alpha:=\{\alpha\zeta^t: t=0,\dots,2l-1\},
\end{array}
$$
where  $\xi,\zeta\in\mathbb{F}$ will denote from now on some fixed $l$th and $2l$th primitive roots of the unit, respectively. 

\begin{teo}\label{teo_clasificacionfinassalvoequivalencia}
Two fine gradings on $H_{n}^\lambda$, $\Gamma=\Gamma(l,s,r;\beta_1,\dots,\beta_s,\alpha_1,\dots,\alpha_r)$
and $\Gamma'=\Gamma(l',s',r';\beta'_1,\dots,\beta'_s,\alpha'_1,\dots,\alpha'_r)$,
are equivalent if and only if $\,l=l'$, $s=s'$, $r=r'$ and there are $\varepsilon\in \mathbb{F}^\times$, $\eta\in S_s$ and $\sigma\in S_r$
such that for all $j=1,\dots,s$,
$\widehat{\varepsilon\beta}'_j=\widehat\beta_{\eta(j)}$ if $l$ is odd and
$\widetilde{\varepsilon\beta}_j'=\widetilde\beta_{\eta(j)}$ if $l$ is even,
and for all $i=1,\dots,r$,
$\widehat{\varepsilon\alpha}'_i=\widehat\alpha_{\sigma(i)}$.
\end{teo}

\begin{proof}
The grading $\Gamma$ is given by
   blocks $B_l^{\textrm{I}}(u,\beta_j)=\{x_1^j,y_1^j,\dots,x_l^j,y_l^j\}$ if $j\le s$,
and   blocks $B_{\frac
l2}^{\textrm{II}}(u,\alpha_i)=\{a_1^i,\dots,a_l^i\}$ if $i\le r$. Also the grading $\Gamma'$ is given by
   blocks $B_{l'}^{\textrm{I}}(u',\beta'_j)$ if $j\le s'$,
and   blocks $B_{\frac
{l'}2}^{\textrm{II}}(u',\alpha'_i)$ if $i\le r'$, for some homogeneous element $u'\in L$ ($z$ is fixed).

Let $f\colon H_{n}^\lambda\to H_{n}^\lambda=L$ be an automorphism applying any homogeneous component of $\Gamma$ into a homogeneous component of $\Gamma'$ (take into account that all of them are one-dimensional). Note that $f$ applies $\mathcal{Z}(L)=\span{z}$ into itself. We can assume that $f(z)=z$, by replacing $f$ with the composition of $f$
 with the automorphism given by
\begin{equation}\label{eq_automorfparaajustarlaz}
z\mapsto \alpha^2z,\,u\mapsto u,\,x_i^j\mapsto\alpha
x_i^j,\,y_i^j\mapsto\alpha y_i^j,\,a_i^t\mapsto\alpha a_i^t
\end{equation}
for a convenient  $\alpha\in \mathbb{F}^{\times}$.

 As $f$ leaves $H_n=[L,L]$ invariant, the homogeneous element $ f(u)\notin f([L,L])$, so that there is $\varepsilon\in \mathbb{F}^{\times}$ such that $f(u)=\varepsilon u'$.
Fixed  $j\le s$,  we have that
$$
f(x_1^j)\in f(B_l^{\textrm{I}}(u,\beta_j))=B_l^{\textrm{I}}(f(u),\beta_j)=B_l^{\textrm{I}}(u',\beta_j/\varepsilon),
 $$
but the homogeneous element $
f(x_1^j)$ must also belong either to some   $\span{B_{l'}^{\textrm{I}}(u',\beta'_q)}$ or to some $\span{B_{\frac
{l'}2}^{\textrm{II}}(u',\alpha'_q)}$, and so it is multiple of one of the   independent elements forming such blocks (by homogeneity). Observe that the second possibility can be ruled out, because any element $x$ in a block of type I (related to $u'$) verifies that  $[x,\varphi^m(x)]=0$ for all  $m\in \mathbb{N}$ (being $\varphi=\ad(u')$), while
 any element $x$ in a block of type II (related to $u'$) verifies that  there is $m\in \mathbb{N}$ such that $[x,\varphi^m(x)]=0$.

 Necessarily $l'=l$, since the size $m$ of a concrete block of type I can be computed by taking $x$ any of its elements (obviously no problem if we take a nonzero multiple) and then $m\ge1$ is the minimum integer such that $\varphi^m(x)\in\span{x}$. As $\varepsilon(\ad u')f=f(\ad u)$, the result follows.

 The also homogeneous element $f(y_1^j)$ must belong to the same   $\span{B_{l}^{\textrm{I}}(u',\beta'_q)}$, since $[ f(x_1^j),f(y_1^j) ]\ne0$ while the bracket of two different blocks is zero. And $f(x_{i+1}^j)=\frac1{\beta_j}\varphi(f(x_i^j))$   belongs to  $\span{B_{l}^{\textrm{I}}(u',\beta'_q)}$ too
 for induction on $i$, since the space spanned by the block is $\varphi$-invariant. We have proved that $\span{B_l^{\textrm{I}}(u',\beta_j/\varepsilon)}=\span{\{f(x_i^j),f(y_i^j):i=1,\dots,l\}}
 \subset\span{B_{l}^{\textrm{I}}(u',\beta'_q)}$, but both spaces have the same dimension $2l$, so they coincide. Now we apply Lemma~\ref{le_cambioenlosbloques}(i) to get that
 $\widehat{\varepsilon\beta}'_q=\widehat\beta_{j}$ if $l$ is odd and
$\widetilde{\varepsilon\beta}_q'=\widetilde\beta_j$ if $l$ is even.

We proceed analogously with the blocks of type II.
\end{proof}

\subsection{Weyl groups}

Finally, we would like to compute the Weyl groups of the fine gradings on the twisted Lie algebra $L=H_n^\lambda$.

 Let $\Gamma$ be a fine grading on $L$. Then, taking into account Theorem~\ref{teo_clasificacionfinassalvoequivalencia},
 there are $l,s,r\in \mathbb{Z}_{\ge0}$ with $n-2=l(2s+r)$  such that
 $\Gamma=\Gamma(l,s,r;\beta_1,\dots,\beta_s,\alpha_1,\dots,\alpha_r)$,
 where the scalars $\beta_j,\alpha_i\in\mathbb{F}^{\times}$ can be chosen such that
 \begin{equation}\label{eq_elrepartoparalosupsilon}
 \begin{array}{ll}
 \{\beta_1,\dots,\beta_s\}=\{\delta_1,\dots,\delta_1,\dots,\delta_{s'},\dots,\delta_{s'}\}&\text{each $\delta_i$ repeated $m_i$ times,}\\
 \{\alpha_1,\dots,\alpha_r\}=\{\gamma_1,\dots,\gamma_1,\dots,\gamma_{r'},\dots,\gamma_{r'}\}&\text{each $\gamma_i$ repeated $n_i$ times},
 \end{array}
 \end{equation}
 (reordering the blocks) and such that
 $$
 \begin{array}{ll}
 \widehat\delta_i\ne\widehat\delta_j,\ \widehat\gamma_i\ne\widehat\gamma_j\quad\forall i\ne j&\text{ if $l$ is even},\\
 \widetilde\delta_i\ne\widetilde\delta_j\quad\forall i\ne j&\text{ if $l$ is odd (so $r=0$)}.
 \end{array}
 $$
 (Obviously $m_1+\dots+m_{s'}=s$ and $n_1+\dots+n_{r'}=r$).

 A basis of homogeneous elements of this grading is formed by $z\in \mathcal{Z}(L)$, $u$,
  blocks $B_l^{\textrm{I}}(u,\beta_j)=\{x_1^j,y_1^j,\dots,x_l^j,y_l^j\}$ if $j\le s$,
and   blocks $B_{\frac
l2}^{\textrm{II}}(u,\alpha_i)=\{a_1^i,\dots,a_l^i\}$ if $i\le r$.

First observe that each $\sigma=(\sigma_1,\dots,\sigma_{r'})\in S_{n_{1 }}\times\dots\times S_{n_{r' }}$
 (that is, $\sigma(n_1+\dots+n_j+t)=n_1+\dots+n_j+\sigma_{j+1}(t)$ if $1\le t\le n_{j+1}$) and each $ \eta\in
 S_{m_{1 }}\times\dots\times S_{m_{s' }}$ allow to define the automorphism $\Upsilon_{(\eta,\sigma)}\in\aut(\Gamma)$ by
 $$
 z\mapsto z,\,
 u\mapsto u,\,
 x_i^j\mapsto x_i^{ \eta(j)},\,
 y_i^j\mapsto y_i^{\eta(j) },\,
 a_i^t\mapsto a_i^{ \sigma(t)},
$$
which obviously preserves the grading but interchanges \emph{the blocks}.

We have  besides  some remarkable elements in the group of automorphisms of
the grading which fix all the spaces spanned by the blocks (generalizing Equations~(\ref{eq_WeylbloquetipoI})
and  (\ref{eq_WeylbloquetipoII})). For each
$j\le s$, consider $\theta_j\in\aut(\Gamma)$ leaving invariant
$\span{\{z,u,x_i^t,y_i^t,a_i^p: i\le l,p\le r,t\le s,t\ne j\}}$
and such that
$$
\theta_j(x_i^j)=\I x_{i+1}^j,\qquad \theta_j(y_i^j)=\I y_{i-1}^j,
$$
(indices taken modulo $l$). If $l$ is   even, we can
also consider $\vartheta _j\in\aut(\Gamma)$ leaving invariant
$\span{\{z,u,x_i^t,y_i^t,a_i^p: i\le l,p\le r,t\le s,t\ne j\}}$
and such that
$$
\vartheta_j(x_i^j)=y_{i}^j,\qquad \vartheta_j(y_i^j)=-x_{i}^j.
$$
Finally consider (also for $l$ even) for each $t\le r$ the automorphism
  $\varrho_t\in\aut(\Gamma)$ leaving invariant the subspace
  $\span{\{z,u,x_i^j,y_i^j,a_i^p: i\le l,p\le r,j\le s,p\ne t\}}$
and such that
$$
\varrho_t(a_i^t)=a_{\frac{l}2+i}^t
$$
for all $i\le l$ (sum modulo $l$).

\begin{lemma}\label{le_losquefijanelu}
If $l$ is even and the automorphism $f\in\aut(\Gamma)$ is such that $f(z)=z$ and $f(u)=u$, then   $[f]$ belongs to the group  generated by
 \begin{equation}\label{eq_losquefijanu}
   \{ [\Upsilon_{(\eta,\sigma)}],[\theta_j],[\vartheta_j],[\varrho_t]: j\le s, t\le r\}
 \end{equation}
with $\eta\in S_{m_{1 }}\times\dots\times S_{m_{s' }}$
 and $\sigma\in S_{n_{1 }}\times\dots\times S_{n_{r' }}$   permutations as above.
\end{lemma}

\begin{proof}
Fixed $j=1,\dots,r$, there is $p=1,\dots,s'$ such that $\delta_p=\beta_j$ ($j=m_1+\dots+m_{p-1}+b$ for some $1\le b\le m_p$). The image of this $j$th block is  $f(B_l^{\textrm{I}}(u,\delta_p))=B_l^{\textrm{I}}(f(u),\delta_p)=B_l^{\textrm{I}}(  u,\delta_p )$, so it must generate the
subspace spanned by the $q$th block for some $q$ such that $\delta_p=\beta_q$ (taking into account Lemma~\ref{le_cambioenlosbloques} and the conditions required after Equation~(\ref{eq_elrepartoparalosupsilon})). In other words,
there is $\eta\in S_s$ such that $f$ applies the $j$th block of type I into the span of the $\eta(j)=q$th block of type I,
with this  $\eta\in S_{m_{1 }}\times\dots\times S_{m_{s' }}$. In the same way, there is
$\sigma\in S_{n_{1 }}\times\dots\times S_{n_{r' }}$ such that any $i$th block of type II is applied into the span of the $\sigma(i)$th block of type II. Hence, the automorphism $\Upsilon_{(\eta,\sigma)}^{-1}f$ applies each block
into the subspace spanned by itself. Now take into account that \emph{inside each block}, we had computed the Weyl groups in   Equations~(\ref{eq_WeylbloquetipoI})
and  (\ref{eq_WeylbloquetipoII}).
\end{proof}

Moreover, note that $\span{[\theta_j],[\vartheta_j]}\cong D_l$ for each fixed $j$, and these elements commute
with all $[\theta_i],[\vartheta_i],[\varrho_t]$ if $i\ne j$ (any $t$). The other elements multiply as follows:
$$
\Upsilon_{(\eta,\sigma)}\theta_j\Upsilon_{(\eta,\sigma)}^{-1}=\theta_{\eta(j)},\quad
\Upsilon_{(\eta,\sigma)}\vartheta_j\Upsilon_{(\eta,\sigma)}^{-1}=\vartheta_{j},\quad
\Upsilon_{(\eta,\sigma)}\rho_t\Upsilon_{(\eta,\sigma)}^{-1}=\rho_{\sigma(t)},
$$
so that their classes do not commute in general.
Thus, the group $\mathcal{W}'$ generated by the set in Equation~(\ref{eq_losquefijanu}) is isomorphic to
$$
\mathcal{W}'\cong (S_{n_{1 }}\times\dots\times S_{n_{r' }}\times S_{m_{1 }}\times\dots\times S_{m_{s' }})\ltimes
 (D_l^s\times \mathbb{Z}_2^r).
 $$

 The following technical lemma will be useful, too.

 \begin{lemma}\label{le_lodeXeYydescomposicionyautomorfimo}
  Under the hypothesis above,
 \begin{itemize}
 \item[a)] If $f\in\aut(L)$ verifies that $f(z)=z$ and $f(u)=\varepsilon u$,
 then $\varepsilon$ is a primitive $p$th root of the unit and the spectrum of $ \ad(u)$ can be splitted as
 \begin{equation}\label{eq_ladescomposicionenXeY}
 \{\widehat\alpha_1,\dots,\widehat\alpha_r\}=\cup_{t=0}^{p-1}\varepsilon^tX,\quad\text{(disjoint union)}\quad
 \{\widehat\beta_1,\dots,\widehat\beta_s\}=\cup_{t=0}^{p-1}\varepsilon^tY,
 \end{equation}
 for some
  sets $X$ and $Y$ of classes such that
 $X\cap\varepsilon^jX$ is either $\emptyset$ or $X$ for any $j$, and in the same way
 $Y\cap\varepsilon^jY$ is either $\emptyset$ or $Y$.
  \item[b)] 
  If there is $p$ a divisor of $r$ and $s$ such that $\beta_{j}=\beta_{j+\frac sp}$ for all $1\le j\le\frac sp$
  and $\alpha_{i}=\alpha_{i+\frac rp}$ for all $1\le i\le\frac rp$,
  then the map $g_p$ given by
\begin{equation}\label{eq_elgsubp}
 z\mapsto z,\,
 u\mapsto\epsilon^{-1}u,\,
 x_i^j\mapsto x_i^{j+\frac sp},\,
 y_i^j\mapsto y_i^{j+\frac sp},\,
 a_i^t\mapsto a_i^{t+\frac rp},\,
 \end{equation}
 if $t\le r$, $i\le l$, $j\le s$, for $\epsilon$  a primitive $p$th root of the unit, is an order $p$ automorphism.
 \end{itemize}
 \end{lemma}

 \begin{proof}
 As $f(B_l^{\textrm{I}}(u,\beta_j))=B_l^{\textrm{I}}(f(u),\beta_j)=B_l^{\textrm{I}}(  u,\beta_j/\varepsilon )$ for all $j$,
by Lemma~\ref{le_cambioenlosbloques} there is $k\le r$ such that $\widehat\beta_j/\varepsilon=\widehat\beta_k$.
Thus there are $\sigma\in S_r$ and $\eta\in
S_s$
 such that
 $$
 \varepsilon\widehat\alpha_i= \widehat\alpha_{\sigma(i)},\
 \varepsilon\widehat\beta_j= \widehat\beta_{\eta(j)}
 $$
 for all $i\le r$ and $j\le s$.
 In particular (multiplying for all the indices), we get that
 $\widehat\varepsilon^r=\widehat1=\widehat\varepsilon^s$ and hence $\varepsilon^{rl}=1$.
 Take $q,p$ the minimum integers such that  $\widehat\varepsilon^q=\widehat1$ and $\varepsilon^p=1$ (so $q$ divides $p$).
 Thus, for all $j$,  the set $\{\widehat\beta_j,\varepsilon\widehat\beta_j,\dots, \varepsilon^{p-1}\widehat\beta_j\}$
 has $q$ different classes repeated $p/q$ times. Now we take $j_1=1$, and by induction $j_{i}\notin\{j_{d},\eta(j_{d}),\dots,\eta^{p-1}(j_{d}):d\le i-1\}$ but such that $\widehat\beta_{j_i}=\widehat\beta_{j_d}$ if there is $d<i$ such that $\widehat\beta_{j_i}\in\{\varepsilon^{g}\widehat\beta_{j_d}:g\le q\}$.
 Hence there is a set  $Y=\{\widehat\beta_{j_1},\dots,\widehat\beta_{j_{s/p}} \}$ of classes such that
 $
 \{\widehat\beta_1,\dots,\widehat\beta_s\}=\cup_{t=0}^{p-1}\varepsilon^tY
 $  (disjoint union)
  which verifies that $  Y\cap\varepsilon^jY=\emptyset$ if $j$ is not a multiple of $q$ and $Y\cap\varepsilon^jY=Y$ otherwise. We can deal with $X$ in the same way.\smallskip

The second item b) is clear. Also it is clear that
given any decomposition as in Equation~(\ref{eq_ladescomposicionenXeY}), we can reorder the blocks to are in the situation of item b) precisely for  $p$ and consider $g_p$ correspondingly.
 \end{proof}

\begin{pr}
Take $p\in\mathbb{N}$ the greatest integer such that, for $\epsilon$ a primitive $p$th root of the unit, there are  $X$ and $Y$ sets of classes such that Equation~(\ref{eq_ladescomposicionenXeY}) holds.
 Let $q$ be the minimum positive integer such that $\widehat\epsilon^q=\widehat1$.

 The Weyl group of $\Gamma$, if $l$ is even, is isomorphic to
 \begin{equation}\label{eq_Weylcasolpar}
 \frac{\left((S_{n_{1 }}\times\dots\times S_{n_{r' }}\times S_{m_{1 }}\times\dots\times S_{m_{s' }})\ltimes
 (D_l^s\times \mathbb{Z}_2^r)\right)
 \rtimes\mathbb{Z}_p}{\mathbb{Z}_{p/q}}.
  \end{equation}
  \end{pr}

 \begin{proof}
 Take  $f \in \aut(\Gamma)$, so that $f(z)\in\span{z}$ and $f(u)\in\span{u}$. Since we  can compose $f$ with an automorphism in the stabilizer  as the given one in
 Equation~(\ref{eq_automorfparaajustarlaz}), then we can assume $f(z)=z$.

 Consider $g_p$   the automorphism described in  Lemma~\ref{le_lodeXeYydescomposicionyautomorfimo}b).
 Note that if we apply Lemma~\ref{le_lodeXeYydescomposicionyautomorfimo}a) to $f$, there is $\varepsilon$ a primitive $p'$th root of the unit  such that $f(u)=\varepsilon u$. So $fg_p^{-1}(u)=\varepsilon\epsilon u$, being $\varepsilon\epsilon$ a root $lcm(p,p')$-primitive. Again by Lemma~\ref{le_lodeXeYydescomposicionyautomorfimo}a), we have the corresponding splitting, so that, by maximality of $p$, we conclude that $lcm(p,p')=p$ and $\varepsilon$ is a power of $\epsilon$. Thus there is $c\in\mathbb{N}$ such that   $f  g_p^c(u)=u$  and thus Lemma~\ref{le_losquefijanelu} is applied and $[f]\in \span{[g_p],\mathcal{W}'}$. Moreover,   $\mathcal{W}(\Gamma)=\mathcal{W}'\rtimes\span{[g_p]}$, since $g_pfg_p^{-1}$ fixes $u$ for any  $f\in \mathcal{W}'$.\smallskip


 In order to make precise the composition of $g_p$ with $  \Upsilon_{(\eta,\sigma)}$, we need to choose the order of the scalars $\beta_j$'s and $\alpha_i$'s adapted to both automorphisms, changing slightly the   order taken  in Equation~(\ref{eq_elrepartoparalosupsilon}). We mean: order the set $Y= \{ \widehat\delta_1,\dots,\widehat\delta_1,\dots,\widehat\delta_{a},\dots,\widehat\delta_{a}\}$ where $\widehat\delta_i$ is repeated $l_i$ times and $\widehat\delta_i\notin\span{\epsilon}\widehat\delta_j$ if $i\ne j$. So the set of different representatives for the blocks of type I are
 $\{\epsilon^t\delta_b: 1\le b\le a,0\le t\le q-1\}$, where $ \epsilon^t\delta_b$ appears $l_b\frac pq$ times in the \emph{positions}
\begin{equation}\label{eq_elconjuntoquehayquedejarfijo}
 \{l_1+\dots+l_{b-1}+i+\frac spt+\frac{sq}pc: i=1,\dots,l_b;\,c=0,\dots,\frac pq-1\}.
 \end{equation}
 (With our notation  $\{m_1,\dots,m_{s'}\}=\cup_{t=0}^{q-1}\{l_b\frac pq: 1\le b\le a\}$.) We denote by  $ S'_s$
 the set of permutations which fix the sets in Equation~(\ref{eq_elconjuntoquehayquedejarfijo}) for all $1\le b\le a,0\le t\le q-1$, which is of course isomorphic to $S_{m_{1 }}\times\dots\times S_{m_{s' }}$. Proceed in the same way with the set $X$ and denote by $\widetilde S_r$ the set of suitable permutations (interchanging the positions corresponding to the same scalars).
 Therefore,
  $$
   \mathcal{W}(\Gamma)=\span{\{[g_p] , [\Upsilon_{(\eta,\sigma)}],[\theta_j],[\vartheta_j],[\varrho_t]: j\le s, t\le r,\eta\in  S'_s,\sigma\in  \widetilde S_r\}},
 $$
where we can observe that if $\eta'(j):=\eta(j+\frac sp)-\frac sp$ and  $\tilde\sigma (t):=\sigma(t+\frac rp)-\frac rp$, then $\eta'\in S'_s$, $\tilde\sigma\in\widetilde S_r$ and
 $$
 g_p^{-1}\Upsilon_{(\eta,\sigma)}g_p=\Upsilon_{(\eta',\tilde\sigma)}.
 $$
 The remaining elements multiply as follows:
 $$
 \begin{array}{l}
 g_p^{-1}\theta_jg_p=\theta_{j-\frac sp},\quad
 g_p^{-1}\vartheta_jg_p=\vartheta_{j-\frac sp},\quad
 g_p^{-1}\rho_tg_p=\rho_{t-\frac rp}.
 \end{array}
 $$
This explains how the
actions in the semidirect products are defined.\smallskip

 Finally,  the new generator $[g_p] $ has order $p$. Besides  $g_p^q(u)=\epsilon^qu\in\span{\xi}u$, so  its composition with some power of the following element in the stabilizer
$$
z\mapsto\xi z,\,u\mapsto\xi u,\,x_i^j\mapsto\xi^{i}
x_i^j,\,y_i^j\mapsto \xi^{i+1} y_i^j,\,a_i^t\mapsto\xi^{i}
a_i^t,
$$
 (followed by a convenient element as the one in  Equation~(\ref{eq_automorfparaajustarlaz})) fixes $u$ (and $z$).
 Thus $[g_p]^q\in \mathcal{W}'$. Moreover,  $q$ is the least number with this property (a power of less order does not stabilize the blocks).
\end{proof}

We aboard now the case $l$   odd. Recall that there are no blocks of type II.
 As above, if $f\in \aut(\Gamma)$
verifies $f(u)= \varepsilon u $, then
 $f(B_l^{\textrm{I}}(u,\beta_j ))=B_l^{\textrm{I}}(  u,\beta_j/\varepsilon )$.
Thus there is $\eta\in S_s$
 such that
 $$
 \varepsilon\widetilde\beta_j=\widetilde\beta_{\eta(j)}
 $$
for all $j=1,\dots,s$, and $\varepsilon$ turns out to be a $p$th root of the
unit for some $p$. So we can divide
$\{\widetilde\beta_1,\dots,\widetilde\beta_s\}=\cup_{t=0}^{p-1}\varepsilon^tY$, with two different pieces disjoint and $p$ maximum such that this decomposition exists. Take $g_p$ as in the  even case.

The maps $\vartheta_j$ are not longer automorphisms, but we can
consider $\vartheta'\in\aut(\Gamma)$ given by
 $$
 \vartheta'(z)=z,\
 \vartheta'(u)=-u,\
 \vartheta'(x_i^j)=(-1)^iy_{i}^j,\
\vartheta'(y_i^j)=(-1)^{i+1}x_{i}^j.
$$
 If $f\in\aut(\Gamma)$   fixes the subspaces spanned by the blocks, it is not difficult to check
  that $f$ belongs to the subgroup generated by
$$
 \{[\theta_j]: j\le s\}\cup \{  [\vartheta']\},
 $$
which is isomorphic to $\mathbb{Z}_l^s\rtimes\mathbb{Z}_2$, and
hence

 \begin{pr}
The Weyl group of $\Gamma=\Gamma(l,s,0;\beta_1,\dots,\beta_s)$, if $l$ is odd, is isomorphic to
\begin{equation}\label{eq_Weylcasolimpar}
 \mathcal{W}(\Gamma)\cong\frac{\left( S_{m_{1 }}\times\dots\times S_{m_{s' }}\times \mathbb{Z}_2\ltimes
  \mathbb{Z}_l^s\right)
 \rtimes\mathbb{Z}_p}{\mathbb{Z}_{p/q}},
  \end{equation}
  with $ m_1+\dots+m_{s'}=s=\frac{n-2}{2l}$, for $m_j$'s  defined as in Equation~(\ref{eq_elrepartoparalosupsilon}),
  $p\in\mathbb{N}$ the greatest integer such that, for $\varepsilon$ a primitive $p$th root of the unit, there is   $Y$ a set  of classes such that Equation~(\ref{eq_ladescomposicionenXeY}) holds, and $q$ minimum such that $\widetilde\varepsilon^q=\widetilde1$.
 \end{pr}

 \begin{example}\rm
 We now describe the Weyl groups of the fine gradings on $L=H_{10}^{(1,1,\I,\I)}$ computed in Example~1.
  Observe first that the results for $\mathcal{W}(\Gamma_1)$ and $\mathcal{W}(\Gamma_2)$ are quite different than those ones  in Proposition~\ref{pr_loscasosqueestansiempre}, since
  $$
  \mathcal{W}(\Gamma_2)\cong\frac{\mathbb{Z}_2^3\rtimes\mathbb{Z}_4}{\mathbb{Z}_2=\span{(\bar1,\bar1,\bar1,\bar2)}},
  \qquad
  \mathcal{W}(\Gamma_1)\cong\frac{\mathbb{Z}_2^6\rtimes\mathbb{Z}_4}{\mathbb{Z}_2}.
  $$
 Indeed, $\mathcal{W}(\Gamma_2)$ has 4 generators:  $[g]\equiv [g_4]$ (the only element with order $4$), $[\vartheta']$ and the classes of the two automorphisms $f_1$ and $f_2$  coming from permutations, such that the three latter ones commute, $[g]$ commute with $[\vartheta']$,
 $[gf_1g^{-1}]=[f_2]$ and $[g]^2=[\vartheta'f_1f_2]$.
And for  $\Gamma_1$ ($(l,r,s)=(2,4,0)$),    the generators of the Weyl group are $\{[\varrho_i]:i=1,\dots,4\}$,
 $[g]\equiv [g_4]$, the automorphism interchanging $e_1$ with $e_2$ and $\hat e_1$ with $\hat e_2$ and the one
 interchanging $e_3$ with $e_4$ and $\hat e_3$ with $\hat e_4$.

 The remaining cases of Example~1 correspond, respectively, to  Weyl groups isomorphic to
 $\mathbb{Z}_2^3\ltimes \mathbb{Z}_2$, $\mathbb{Z}_2^4$, $\mathbb{Z}_2^2$ and $D_4$.
 \end{example}

\medskip

{\bf Acknowledgment.} The authors would like to thank the referee
for his exhaustive  review of the paper as well as for
many suggestions which have helped to improve the work.

\end{document}